\documentclass[12pt,leqno]{amsart}
\usepackage{amsxtra}
\usepackage{amsopn}
\usepackage{color}
\usepackage{amsmath,amsthm,amssymb}
\usepackage{amscd}
\usepackage{amsfonts}
\usepackage[utf8x]{inputenc}
\usepackage{latexsym}
\usepackage{verbatim}
\usepackage{booktabs}
\usepackage{caption}
\usepackage{soul} % para tachar palabras
\usepackage{multicol}

\theoremstyle{plain}
\newtheorem{theorem}{Theorem}[section]

\newtheorem{lemma}[theorem]{Lemma}
\newtheorem{proposition}[theorem]{Proposition}
\newtheorem{corollary}[theorem]{Corollary}
\newtheorem{remark}[theorem]{Remark}

\newtheorem{example}[theorem]{Example}

\newtheorem{remark-question}[section]{Remark-Question}

\newcommand\C{{\mathbb C}}

\newcommand\N{{\mathbb N}}
\newcommand\Q{{\mathbb Q}}
\newcommand\R{{\mathbb R}}

\newcommand\trace{{\rm tr}}
\newcommand\SU{{\rm SU}}
\newcommand\U{{\rm U}}

\newcommand\frg{{\mathfrak g}}
\newcommand\frh{{\mathfrak h}}
\newcommand\frk{{\mathfrak k}}

\newcommand\frz{{\mathfrak z}}

\newcommand{\alt}{\raise1pt\hbox{$\bigwedge$}}
\def\pint{\langle \cdotp,\cdotp \rangle }

\newcommand{\Ric}{\operatorname{Ric}}

\allowdisplaybreaks[1]

\newcommand{\tr}{\operatorname{tr}}
\newcommand{\ad}{\operatorname{ad}}

\begin{document}
	
\begin{abstract}
		The holonomy of the Bismut connection on Vaisman manifolds is studied. We prove that if $M^{2n}$ is endowed with a Vaisman structure, then the holonomy group of the Bismut connection is contained in U$(n-1)$. We compute explicitly this group for particular types of manifolds, namely, solvmanifolds and some classical Hopf manifolds.
\end{abstract}
	
\title{Bismut connection on Vaisman manifolds}
	
\date{\today}
	
\author[A. Andrada]{Adri\'an Andrada}
\address[Andrada]{FaMAF-CIEM, Universidad Nacional de C\'{o}rdoba, Ciudad Universitaria, X5000HUA C\'{o}rdoba, Argentina} \email{andrada@famaf.unc.edu.ar}
	
\author[R. Villacampa]{Raquel Villacampa}
\address[Villacampa]{Centro Universitario de la Defensa Zaragoza-I.U.M.A.,	Academia Gene\-ral Militar\\
		Carretera de Huesca, s/n	50090 Zaragoza, Spain} 
\email{raquelvg@unizar.es}
	
\maketitle
	
\tableofcontents

\section{Introduction}

A Hermitian connection on a Hermitian manifold $(M,J,g)$ is a connection which leaves both $J$ and $g$ parallel. Each Hermitian manifold admits plenty of these connections. Among them, there is only one whose torsion is totally skew-symmetric. This unique connection is called the Bismut connection associated to $(J,g)$, and it is also known as the Strominger connection or the KT connection (for K\"ahler with torsion). In this article we will denote it by $\nabla^b$.

As with any connection, it is important to determine which are the Hermitian manifolds whose corresponding Bismut connection is flat. Well-known examples of such manifolds are given by Lie groups equipped with a bi-invariant Riemannian metric and a compatible left invariant complex structure (see \cite{IP}). In particular, this family contains the Hermitian manifolds $(G,J,g)$ where $G$ is a compact Lie group, $J$ is one of the left invariant complex structures constructed by Samelson in \cite{Sam} and $g$ is a bi-invariant metric. More recently, it was proved in \cite{WYZ} that every compact Bismut-flat Hermitian manifold is closely related to these examples: indeed, if $M$ is a compact Hermitian manifold with flat Bismut connection, then its universal cover is a Lie group $G'$ equipped with a bi-invariant metric and a left invariant complex structure compatible with the metric. In particular, $G'$ is the product of a compact semisimple Lie group and a real vector space.

Since the flat case is already settled, it is interesting to analyze other Hermitian manifolds whose associated Bismut connection have special curvature properties.  One way is to consider the notion of \emph{K\"ahler-like}. In \cite{AOUV}, the authors conjectured that if the Bismut connection is K\"ahler-like, then the metric is pluriclosed (i.e., $\partial\overline{\partial} \omega=0$, where $\omega$ denotes the fundamental $2$-form $\omega=g(J\cdot,\cdot)$).  This conjecture has been recently proved in \cite{ZZ}. Another way,  and this is the aim of the paper, is to study the holonomy group $\operatorname{Hol}^b$ of the Bismut connection $\nabla^b$. Since both the complex structure and the Hermitian metric are $\nabla^b$-parallel we have that $\operatorname{Hol}^b\subseteq \operatorname{U}(n)$, where $2n$ is the real dimension of the manifold. In particular, $2n$-dimensional Hermitian manifolds whose Bismut holonomy is contained in $\operatorname{SU}(n)$ have attracted plenty of attention. These manifolds are known as \textit{Calabi-Yau with torsion}, and they appear in heterotic string theory, related to the Strominger system in six dimensions. It has been shown that this reduction to $\operatorname{SU}(n)$ is related in certain cases to the Hermitian metric being \textit{balanced}, that is, when the fundamental $2$-form $\omega$ satisfies $d\omega^{n-1}=0$ or, equivalently, $d^\ast\omega=0$.
For instance, it was shown in \cite{St,LY} that if the compact Hermitian manifold $(M^{2n},J,g)$ has holomorphically trivial canonical bundle, then $\operatorname{Hol}^b\subseteq \operatorname{SU}(n)$ if and only if $g$ is conformally balanced; in particular, $(M,J)$ admits a balanced metric. In the case when $M^{2n}$ is a nilmanifold, that is, $M=\Gamma\backslash G$ where $G$ is a nilpotent Lie group and $\Gamma$ is a co-compact discrete subgroup of $G$, more can be said, since it was proved in \cite{FPS} that an invariant Hermitian structure $(J,g)$ on $M$ satisfies $\operatorname{Hol}^b\subseteq \operatorname{SU}(n)$ if and only if $g$ is balanced. 

In the Gray-Hervella classification of almost Hermitian structures, balanced metrics fall into the class $\mathcal{W}_3$. In this article we are interested in Hermitian manifolds which belong to the class $\mathcal{W}_4$, namely \textit{locally conformally K\"ahler} manifolds (or LCK for short). As the name suggests, these Hermitian manifolds are characterized by the property that each point has a  neighbourhood where the metric is conformal to a K\"ahler metric. This condition is equivalent to the existence of a closed $1$-form $\theta$ satisfying $d\omega=\theta\wedge \omega$. The $1$-form $\theta$ is known as the Lee form, and it is given by $\theta=-\frac{1}{n-1} d^\ast\omega\circ J$, where $2n$ is the real dimension of the manifold. A distinguished class of LCK manifolds is given by those where the Lee form is parallel with respect to the Levi-Civita connection of the Hermitian metric. These manifolds were first studied by I. Vaisman in the late '70s (see for instance \cite{V}) and, accordingly, they are nowadays known as \textit{Vaisman} manifolds. Not all LCK manifolds are Vaisman, for instance, the Oeljeklaus-Toma manifolds of type $(s,1)$ are compact complex manifolds which admit LCK metrics but do not admit any Vaisman metric (see \cite{OT,K}).  

Our main goal is to study the Bismut holonomy of Vaisman manifolds and exhibit explicit examples where this holonomy can be computed. We point out that the Riemannian holonomy of compact Vaisman manifolds has been analyzed in \cite{MMP}. Examples of Vaisman manifolds are  given by the classical Hopf manifolds, that is, quotients of $\C^{n}-\{0\}$ by a group of automorphisms generated by $z\to \lambda z$, where $\lambda$ is a complex number with $|\lambda|>1$. These manifolds are all diffeomorphic to $S^1\times S^{2n-1}$ and do not admit any K\"ahler structure.

Another family of examples of compact Vaisman manifolds was introduced in \cite{CFL} in 1986. They are defined as compact quotients of the nilpotent Lie groups $H_{2n+1}\times\R$ by a discrete subgroup $\Gamma$, where $H_{2n+1}$ denotes the $(2n+1)$-dimensional Heisenberg Lie group; they are thus examples of \textit{nilmanifolds}. In these examples the Vaisman structures are left invariant, and recently Bazzoni proved in \cite{Ba} that if a nilmanifold $\Gamma\backslash N$ admits a Vaisman structure (invariant or not) then $N$ is isomorphic to $H_{2n+1}\times\R$. Examples of Vaisman structures on solvmanifolds (i.e., compact quotients of a simply connected solvable Lie group by a discrete subgroup) first appeared in \cite{MP} in 1997. More recently, there have been advances on the structure of the Lie algebras associated to solvmanifolds equipped with invariant Vaisman structures, see for instance \cite{AHK,AO}. In particular, the description given in \cite{AO} will be very useful for us in order to analyze the Bismut connection on these Vaisman solvmanifolds and compute its holonomy.

The paper is structured as follows: In Section 2 we start collecting some known results about Gauduchon connections, holonomy and the Ambrose-Singer theorem as a tool for determining the holonomy.  In Section 3 we study the Bismut connection on Vaisman manifolds and its curvature. The main results are Corollary~\ref{hol} and Corollary~\ref{parallel torsion} where we prove that the holonomy of the Bismut connection on Vaiman manifolds of real dimension $2n$ reduces to $\operatorname{U}(n-1)$ and that the Bismut torsion 3-form is $\nabla^b$-parallel. Section 4 is devoted to solvmanifolds endowed with an invariant Vaisman structure. In this setting, we prove that the holonomy of the Bismut connection has dimension 1 and it is not contained in $\operatorname{SU}(n)$. Some classical Hopf manifolds are studied in Section 5. Using a global parallelization of these manifolds which is compatible with the Vaisman structure, we determine explicitly the holonomy group of the Bismut connection, obtaining the group $\operatorname{U}(n-1)$. Non-Vaisman LCK Oeljeklaus-Toma manifolds are considered in Section 6 and we show that there is no reduction of the Bismut holonomy in this case.  Finally, we study the parallelism of the Lee form $\theta$ of a Vaisman manifold for the line of Gauduchon connections and more generally, for the $2$-parameter family of metric connections introduced in \cite{OUV}.

All manifolds considered in this paper have real dimension $\geq 4$.

\ 

\section{Preliminaries on holonomy and the Ambrose-Singer theorem}

We collect here some well-known facts on holonomy groups and the Ambrose-Singer theorem that will be useful in subsequent sections.

\medskip

Let $\nabla$ denote any linear connection on a connected manifold $M$ and let us fix a point $p\in M$. If $\gamma:[0,1]\to M$ is a piecewise smooth loop based at $p$, the connection $\nabla$ gives rise to a parallel transport map $P_\gamma:T_pM \to T_pM$, which is linear and invertible. The holonomy group  of $\nabla$ based at $p\in M$ is defined as 
\[ \operatorname{Hol}_p(\nabla)=\{ P_\gamma\in \operatorname{GL}(T_pM) \mid \gamma \text{ is a loop based at } p\}.\] 
It turns out that $\operatorname{Hol}_p(\nabla)$ is a Lie subgroup of $\operatorname{GL}(T_pM)$. 
Since $M$ is connected, the holonomy groups based at two different points are conjugated, and therefore we can speak of the holonomy group of $\nabla$, denoted simply by $\operatorname{Hol}(\nabla)$. If $\dim M=n$, we can identify $\operatorname{Hol}(\nabla)$ with a Lie subgroup of $\operatorname{GL}(n,\R)$, after some choice of basis.
The holonomy group need not be connected, and its identity component is denoted by $\operatorname{Hol}_0(\nabla)$; it is known as the restricted holonomy group of $\nabla$ and it consists of the parallel transport maps $P_\gamma$ where $\gamma$ is null-homotopic. Clearly, if $M$ is also simply connected then $\operatorname{Hol}(\nabla)=\operatorname{Hol}_0(\nabla)$.

We point out that if $\nabla$ is a metric connection on a Riemannian manifold $(M,g)$, i.e. $\nabla g=0$, then $P_\gamma$ is an isometry of $(T_pM,g_p)$, while if $\nabla$ satisfies $\nabla J=0$ on an almost complex manifold $(M,J)$ then $P_\gamma J=JP_\gamma$. Therefore, if $\nabla$ is a Hermitian connection ($\nabla g=\nabla J=0$) on an almost Hermitian manifold $(M^{2n},J,g)$ then 
\[ \operatorname{Hol}(\nabla)\subseteq \operatorname{O}(n)\cap \operatorname{GL}(n,\C)=\operatorname{U}(n).\]

%Let $(M,g,J)$ be a connected Hermitian manifold with fundamental 2-form $\omega$ defined by $\omega(X,Y)=g(JX,Y)$. 	A linear connection $\nabla$ on $M$ is called Hermitian if $\nabla g=0$ and $\nabla J=0$. Clearly a Hermitian connection $\nabla$ is torsion-free if and only if the Hermitian structure $(g,J)$ is K\"ahler and $\nabla=\nabla^g$, the Levi-Civita connection of $g$. 

	%The complex structure $J$ extends to $r$-forms as
	%\[ J\alpha(X_1,\ldots,X_r)=(-1)^r\alpha(JX_1,\ldots,JX_r),\quad \alpha\in\Omega^r(M).\]
	
In \cite{Ga} a monoparametric family $\{\nabla^t\}_{t\in\R}$ of Hermitian connections on any Hermitian manifold $(M,g,J)$ was introduced. These are known as the \textit{Gauduchon} connections (or \textit{canonical} connections), and they can be written as
\[ g(\nabla^t_XY,Z)=g(\nabla^g_XY,Z)+\frac{t-1}{4}(d^c\omega)(X,Y,Z)+\frac{t+1}{4}(d^c\omega)(X,JY,JZ), \quad X,Y,Z\in\mathfrak{X}(M), \]
where $\omega$ is the fundamental 2-form  $\omega(X,Y)=g(JX,Y)$ and $d^c:\Omega^r(M)\to\Omega^{r+1}(M)$ is the operator defined by
$d^c=(-1)^rJdJ$. More explicitly, for $\alpha\in\Omega^2(M)$ we have that $d^c\alpha(U,V,W)=-d\alpha(JU,JV,JW)$ for any $U,V,W\in\mathfrak{X}(M)$, so that the expression for $\nabla^t$ becomes
	\begin{equation}\label{canonical} 
	g(\nabla^t_XY,Z)=g(\nabla^g_XY,Z)-\frac{t-1}{4}d\omega(JX,JY,JZ)-\frac{t+1}{4}d\omega(JX,Y,Z). 
	\end{equation}
When $(M,J,g)$ is K\"ahler this family of connections reduces to a single point, given by the Levi-Civita connection. However, in general, the torsion $T^t$ of these connections is non-zero, where $T^t$ is the $(1,2)$-tensor defined by $T^t(X,Y)=\nabla^t_XY-\nabla^t_YX-[X,Y]$ for $X,Y$ vector fields on $M$. 

For particular values of $t\in\R$ we obtain well-known Hermitian connections. For instance, for $t=1$ we have the \textit{Chern connection}, while for $t=0$ we have the \textit{first canonical connection}.

\medskip

In this article we will focus on the \textit{Bismut connection}, which is the connection $\nabla^{-1}$ obtained for $t=-1$. From now on the Bismut connection will be denoted by $\nabla^b$, with corresponding torsion $T^b$. It was introduced in \cite{Bis} and it can be defined as the unique Hermitian connection whose torsion $T^b$ is totally skew-symmetric, i.e. $c(X,Y,Z):=g(X,T^b(Y,Z))$ is a $3$-form on $M$. It follows from \eqref{canonical} that its expression is given by 
\begin{equation} \label{bismut}
	g(\nabla^b_XY,Z)=g(\nabla^g_XY,Z)+\frac12 d\omega(JX,JY,JZ),
\end{equation}
and its torsion $3$-form $c$ is:
\begin{equation} \label{3-form}
		c(X,Y,Z)=d\omega(JX,JY,JZ), 
	\end{equation}
for any $X,Y,Z\in\mathfrak{X}(M)$.

\medskip 

As notation, we will use $\operatorname{Hol}^b(M)$ to refer to the holonomy of the Bismut connection on the Hermitian manifold $M$.

\ 

The Ambrose-Singer theorem provides a way to compute the holonomy group of a linear connection; indeed, it describes the Lie algebra $\mathfrak{hol}_p(\nabla)$ of $\operatorname{Hol}_p(\nabla)$ in terms of curvature endomorphisms $R_p(x,y)$ for $x,y\in T_pM$:

\begin{theorem}\label{AS1}\cite{AS}
The holonomy algebra $\mathfrak{hol}_p(\nabla)$ is the smallest subalgebra of $\mathfrak{gl}(T_pM)$ containing the endomorphims $P_\sigma^{-1}\circ R_p(x,y)\circ P_\sigma$, where $x,y$ run through $T_pM$, $\sigma$ runs through all piecewise smooth paths starting from $p$ and $P_\sigma$ denotes the parallel transport map along $\sigma$.
\end{theorem}

In particular, $\mathfrak{hol}_p(\nabla)$ contains all the curvature endomorphisms $R_p(x,y),\, x,y\in T_pM$. This fact will be used in Section \ref{Hopf}.

\medskip

Let us consider the particular case when $M=G$ is a Lie group with Lie algebra $\frg$. A linear connection $\nabla$ on $G$ is said to be left invariant if the left translations on $G$ are affine maps. As a consequence, if $X,Y$ are left invariant vector fields then $\nabla_XY$ is also left invariant. Therefore $\nabla$ is uniquely determined by a bilinear multiplication $\frg \times \frg \to \frg$, still denoted by $\nabla$. We also denote by $\nabla_x:\frg\to\frg$ the endomorphism defined by left multiplication with $x\in\frg$. In this case the Ambrose-Singer theorem takes the following form:

\begin{theorem}\label{AS2}\cite{Alek}
Let $\nabla$ be a left invariant linear connection on the Lie group $G$, and let $\frg$ denote the Lie algebra of $G$. Then the holonomy algebra $\mathfrak{hol}(\nabla)$, based at the identity element $e\in G$, is the smallest subalgebra of $\mathfrak{gl}(\frg)$ containing the curvature endomorphisms $R(x,y)$ for any $x,y\in\frg$, and closed under commutators with the left multiplication operators $\nabla_x:\frg\to \frg$.
\end{theorem}

This version of the Ambrose-Singer theorem will be used in Sections \ref{solvmanifold} and \ref{section-OT}.

\

\section{Curvature of the Bismut connection on Vaisman manifolds}

%\subsection{Basic results about the Bismut connection on LCK manifolds}

Let $(J,g)$ be a Hermitian structure on a connected manifold $M$ with fundamental $2$-form~$\omega$. This structure is called \textit{locally conformally K\"ahler} (LCK for short) if there exists an open covering $\{ U_i\}_{i\in I}$ of $M$ and differentiable functions $f_i:U_i \to \R$, $i\in I$, such that each local
metric $g_i=\exp(-f_i)\,g|_{U_i}$ is K\"ahler. 
Equivalently, $(J,g)$ is LCK if there exists 
a closed $1$-form $\theta$ such that the differential of $\omega$ is given by 
\begin{equation}\label{dif} 
    d\omega=\theta\wedge \omega.
\end{equation} 
The $1$-form $\theta$ is known as the Lee form. We denote by $A\in\mathfrak{X}(M)$ the vector field which is metric dual to $\theta$, i.e., $g(A,U)=\theta(U)$ for all $U\in\mathfrak{X}(M)$. If, moreover, $\nabla^g\theta=0$, the Hermitian structure $(J,g)$ is called \textit{Vaisman}.

\ 

\begin{remark}
{\rm (i) The Lee form is uniquely determined by
\begin{equation} \label{d-theta}
\theta=-\frac{1}{n-1}(d^\ast\omega)\circ J,
\end{equation}
where $d^\ast$ is the codifferential and $2n$ is the real dimension of $M$. 

\smallskip (ii) If the Lee form $\theta$ is exact, i.e. $\theta=df$ with $f\in C^\infty(M)$, then $\exp(-f)g$ is a K\"ahler metric on $M$. Therefore any simply connected LCK manifold admits a global K\"ahler metric; consequently, ``genuine" LCK metrics occur on non-simply connected manifolds.

\smallskip (iii) The LCK structure is K\"ahler if and only if $\theta=0$. Indeed, $\theta \wedge \omega=0$ and $\omega$ non degenerate imply $\theta=0$.}
\end{remark}

\ 
	
Let us start with some results concerning the torsion of the Bismut connection on an LCK manifold. 

\medskip

\begin{lemma}\label{torsion}
Let $(M,J,g)$ be an LCK manifold with fundamental $2$-form $\omega$ and Lee form~$\theta$. Then the torsion $3$-form $c$ of the Bismut connection is given by 
	\[   c=-J\theta \wedge \omega,\]
where $J\theta$ denotes the $1$-form on $M$ defined by $J\theta(X)=-\theta(JX)$.
\end{lemma}
	
\begin{proof}
Recall that the torsion $3$-form $c$ is given by $c(X,Y,Z)=d\omega(JX,JY,JZ)$. It follows from \eqref{dif} that
		\begin{align*}
			c(X,Y,Z) & = \theta\wedge \omega(JX,JY,JZ) \\
			& = \theta(JX)\omega(JY,JZ)+\theta(JY)\omega(JZ,JX)+\theta(JZ)\omega(JX,JY) \\
			& = -J\theta(X)\omega(Y,Z)-J\theta(Y)\omega(Z,X)-J\theta(Z)\omega(X,Y) \\
			& = -J\theta\wedge \omega(X,Y,Z),
		\end{align*}
for any vector fields $X,Y,Z$ on $M$.
\end{proof}
		
\begin{corollary}\label{torsion-A}
The torsion $3$-form $c$ of the Bismut connection satisfies $c(A,\cdot,\cdot) = 0$. 
\end{corollary}
	
\begin{proof} Using the previous expression for the torsion, one gets:
	\begin{align*}
		c(A,X,Y) &=	(J\theta\wedge \omega)(A,X,Y)  \\  & = -\theta(JA)\omega(X,Y) -\theta(JX)\omega(Y,A)-\theta(JY)\omega(A,X)\\
			& =-g(A,JA)g(JX,Y) -g(A,JX)g(JY,A)-g(A,JY)g(JA,X) \\
			& = 0, 
	\end{align*} since $g(A,JA)=0$ and $J$ is skew-symmetric. 
\end{proof}
	
\begin{corollary}\label{torsion-12}
The $(1,2)$-torsion tensor $T^b$ of the Bismut connection is given by
		\[  T^b(X,Y)= \theta(JX)JY-\theta(JY)JX-\omega(X,Y)JA, \]
for any vector fields $X,Y$ on $M$.
\end{corollary}
	
\begin{proof}
Recalling that $c(X,Y,Z)=c(Z,X,Y)=g(Z,T^b(X,Y))$, this follows easily from Lemma \ref{torsion}.
\end{proof}
	
\medskip
	
As a consequence, we have that 
\begin{equation}\label{nablab-formula}
	\nabla^b_XY=\nabla^g_XY+\frac12\left( \theta(JX)JY-\theta(JY)JX-\omega(X,Y)JA \right),
\end{equation}
for any vector fields $X,Y$ on $M$.
	
\

From now on, we focus on Vaisman manifolds. Our objective is to compute the Bismut connection on Vaisman manifolds in terms of the vector fields $A,JA$ and the distribution $\mathcal D$ orthogonal to $A$ and $JA$, and then study the symmetries of the corresponding curvature tensor $R^b$. In particular, we obtain that the torsion $3$-form $c$ is always $\nabla^b$-parallel on a Vaisman manifold, and we obtain a first reduction of its holonomy.

\ 

\subsection{Results about the Bismut connection on Vaisman manifolds}

On a Vaisman manifold, the vector field $A$ $g$-dual to the Lee form $\theta$ satisfies $\nabla^g A=0$. It follows that $|A|$ is constant on $M$ and, therefore, by rescaling the metric, we may assume from now on that $|A|=1$, so that $\theta(A)=1$.

\smallskip

In the next result we collect some well-known facts about Vaisman manifolds which will be used throughout this article. 
	
\begin{proposition}\cite{V}\label{propiedades-Vaisman}
Let $(M,J,g)$ be a Vaisman manifold with associated Lee form $\theta$. Let $A$ be the vector field which is metric dual to $\theta$, with $|A|=1$. Then:
	\begin{enumerate}
		\item[(a)] $[A,JA]=0$;
		\item[(b)] both $A$ and $JA$ are Killing vector fields;
		\item[(c)] $\mathcal{L}_AJ=\mathcal{L}_{JA}J=0$, where $\mathcal{L}$ denotes the Lie derivative. That is, $[A,JX]=J[A,X]$, $[JA,JX]=J[JA,X]$ for any vector field $X$ on $M$.
	\end{enumerate}
\end{proposition}
	
\medskip
	
We denote by $\mathcal D$ the distribution on $M$ such that $\mathcal{D}_p$ is the orthogonal complement of $\text{span}\{A_p,J_pA_p\}$ in $T_pM$ for any $p\in M$. Clearly, $\mathcal{D}$ is $J$-invariant and $\theta(X)=J\theta(X)=0$ for any $X\in\Gamma(\mathcal{D})$. Moreover, $\mathcal D$ is not involutive, since using $d\omega(A,X,Y)=c\wedge \omega(A,X,Y)$ and Proposition \ref{propiedades-Vaisman} it can be seen that $g(JA,[X,Y])=\omega(X,Y)$ for any $X,Y\in\Gamma(\mathcal D)$. However, we can show that   
	
\begin{corollary}\label{D}
If $X\in\Gamma(\mathcal{D})$ then $[A,X]\in\Gamma(\mathcal{D})$ and $[JA,X]\in\Gamma(\mathcal{D})$. 
\end{corollary}
	
\begin{proof}
Since $d\theta=0$,
	\begin{align*}
			0=d\theta(A,X) & = A(\theta(X))-X(\theta(A))-\theta([A,X])\\
			& =Ag(A,X)-Xg(A,A)-g(A,[A,X])\\
			& = -g(A,[A,X]).
	\end{align*}		
If in this expression we replace $X\in\Gamma(\mathcal D)$ by $JX\in\Gamma(\mathcal D)$, and using Proposition \ref{propiedades-Vaisman}, we obtain 
\[ 0=-g(A,[A,JX])=g(JA,[A,X]).\]
 Thus, $[A,X]\in\Gamma(\mathcal{D})$. 

The fact that $[JA,X]\in\Gamma(\mathcal{D})$ follows in the same way from $d\theta(JA,X)=0$.
\end{proof}
	
\
	
We will prove next that the Lee form $\theta$ on an LCK manifold is parallel with respect to the Levi-Civita connection (i.e. the manifold is Vaisman) if and only if it is parallel with respect to the Bismut connection. This fact was already mentioned in \cite{Sc}.
	
\begin{theorem}\label{theta-parallel}
Let $(M,J,g)$ be an LCK manifold with fundamental $2$-form $\omega$ and Lee form~$\theta$. Then $(M,J,g)$ is Vaisman (i.e., $\nabla^g\theta=0$) if and only if $\nabla^b\theta=0$.
\end{theorem}
	
\begin{proof}
Let us compute $(\nabla^b_X \theta)Y$ for any $X,Y\in\mathfrak{X}(M)$:
	\begin{align*}
			(\nabla^b_X \theta)Y & = X(\theta(Y))-\theta(\nabla^b_XY) \\
			& = Xg(A,Y)-g(\nabla^b_XY,A) \\
			& = g(\nabla^g_XA,Y)+g(A,\nabla^g_XY)- \left(g(\nabla^g_XY,A)+\frac12 c(X,Y,A)\right) \\
		%& = g(\nabla^g_XA,Y)+\frac12 (J\theta\wedge \omega)(X,Y,A)\\
			& = g(\nabla^g_XA,Y)
	\end{align*}
	according to %Lemma \ref{torsion} and 
	Corollary \ref{torsion-A}. Since $g(\nabla^g_XA,Y)=(\nabla^g_X \theta)Y$, the result follows.
\end{proof}
	
Since $\nabla^b J=0$, it follows from Theorem \ref{theta-parallel} that $\nabla^b J\theta=0$, or equivalently, $\nabla^b JA=0$. As an immediate consequence we have the following important result:
	
\begin{corollary}\label{hol}
If $(M,J,g)$ is Vaisman and $\dim M=2n$, then the holonomy group $\operatorname{Hol}^b(M)$ of the Bismut connection is contained in $\operatorname{U}(n-1)$.
\end{corollary}
	
Here $\operatorname{U}(n-1)$ is considered as a subgroup of $\operatorname{U}(n)$ in the following way:
\[  \operatorname{U}(n-1)\hookrightarrow \operatorname{U}(n), \qquad A \mapsto\bigg( \begin{array}{c|c}  1& \\ \hline & A \end{array}\bigg).\]
	
\

Also, combining $\nabla^b J\theta=0$ and $\nabla^b \omega=0$ with Lemma \ref{torsion} we obtain 
	
\begin{corollary}\label{parallel torsion}
On any Vaisman manifold	the torsion $3$-form $c$ of the Bismut connection is $\nabla^b$-parallel.
\end{corollary}
	
\medskip
	
\begin{remark}
{\rm The converse of Corollary \ref{parallel torsion} holds: if $(M,J,g)$ is an LCK manifold such that the torsion of the Bismut connection is $\nabla^b$-parallel then $(M,J,g)$ is Vaisman. 
			
Indeed, according to Lemma \ref{torsion} the torsion of $\nabla^b$ is $c=-J\theta\wedge\omega$. If $\nabla^b_X(J\theta\wedge \omega)=0$ for any vector field $X$, then $\nabla^b_X(J\theta)\wedge \omega=0$, since $\nabla^b\omega=0$. Since $\omega$ is non degenerate, the operator $-\wedge \omega$ is injective on $1$-forms, hence $\nabla^b_X(J\theta)=0$. We deduce from Theorem \ref{theta-parallel} that $M$ is Vaisman.}
\end{remark}

\

In order to compute the holonomy of the Bismut connection on a Vaisman manifold, we look first for parallel tensors. Let us consider the following skew-symmetric $(1,1)$-tensor:
\begin{equation}\label{tensor-phi}
    \varphi=J-\theta\otimes JA+J\theta\otimes A.
\end{equation}
This tensor was introduced by Vaisman in \cite{V} and it is an \textit{$f$-structure} (i.e. $\varphi$ satisfies  $\varphi^3+\varphi=0$). It has some important properties related to the Bismut connection, as the following propositions show.

\begin{proposition}\label{phi paralelo}
On any Vaisman manifold, the tensor $\varphi$ is $\nabla^b$-parallel. 
\end{proposition}
	
\begin {proof}
For any $X\in\mathfrak{X}(M)$ we have that 
	\[ \nabla^b_X\varphi= \nabla^b_XJ+\nabla^b_X(-\theta\otimes JA+J\theta\otimes A).\]
Since $\nabla^bJ=0$, we only have to check that the second term vanishes. We compute 
\begin{align*} 
	\nabla^b_X(-\theta\otimes JA+J\theta\otimes A) & = -\nabla^b_X\theta \otimes JA -\theta\otimes \nabla^b_XJA +\nabla^b_X J\theta\otimes A + J\theta\otimes \nabla^b_XA \\
		& = 0,
\end{align*}
using again that $\nabla^b\theta=0$, $\nabla^bA=0$ and $\nabla^bJ=0$.
\end {proof}
	
\medskip

\begin{proposition}\label{torsion-JA}
On any Vaisman manifold, the torsion $3$-form $c$ of the Bismut connection satisfies $c(JA,X,Y) = -g(\varphi(X),Y)$.
\end{proposition}

\begin{proof}
For $X,Y\in\mathfrak{X}(M)$, using Proposition \ref{torsion} we have that
\begin{align*}
    c(JA,X,Y) & = -J\theta \wedge \omega (JA,X,Y) \\
              & =  -J\theta(JA)\omega(X,Y)-J\theta(X)\omega(Y,JA)-J\theta(Y)\omega(JA,X) \\
			& = -g(JX,Y)-J\theta(X)g(Y,A)-g(A,JY)g(A,X) \\
			& = -g(JX,Y)-J\theta(X)g(A,Y)+\theta(X)g(JA,Y) \\		
			& = -g(\varphi(X),Y),
\end{align*}
and the proof is complete.
\end{proof}

\medskip

The tensor $\varphi$ is closely related to the $2$-form $d(J\theta)$, as the following result shows:

\begin{corollary}\label{coro-phi}
The $2$-form $d(J\theta)$ satisfies:
    \begin{enumerate}
        \item[(a)] $d(J\theta)(X,Y)=c(JA,X,Y)$ (hence also equal to $-g(\varphi(X),Y)$),
        \item[(b)] $d(J\theta)$ is $\nabla^b$-parallel,
        \item[(c)] $d(J\theta)(JX,JY)=d(J\theta)(X,Y)$ for any $X,Y\in\mathfrak{X}(M)$,
        \item[(d)] $d(J\theta)(A,\cdot)=d(J\theta)(JA,\cdot)=0$.
    \end{enumerate}
\end{corollary}

\begin{proof}
For $X,Y\in\mathfrak{X}(M)$, we compute
\begin{align*}
		d(J\theta)(X,Y) & = X(J\theta (Y))-Y(J\theta(X))-J\theta([X,Y]) \\
	            	            & = -X(\theta(JY))+Y(\theta(JX))+\theta(J[X,Y]) \\
								& = -Xg(A,JY)+Yg(A,JX)+g(A,J[X,Y]). 
\end{align*}
Since $\nabla^bg=0$, we have that
		\[  d(J\theta)(X,Y) = -g(\nabla^b_XA,JY)-g(A,\nabla^b_X JY)+g(\nabla^b_YA,JX)+g(A,\nabla^b_Y JX)+g(A,J[X,Y]). \]
According to Theorem \ref{theta-parallel}, we have that $\nabla^bA=0$. Now, using that $\nabla^bJ=0$ and $J$ is skew-symmetric, we obtain that
		\[  d(J\theta)(X,Y) = g(JA,T^b(X,Y))=c(JA,X,Y). \]
Therefore (a) holds. Now, (b) follows immediately from Proposition \ref{torsion-JA}.

Finally, (c) and (d) follow readily from (a). Indeed, for (c) we use that $\varphi$ and $J$ commute, and for (d) we use that $\varphi(A)=\varphi(JA)=0$.
\end{proof}

\
	
\subsection{Explicit computation of $\nabla^b$}
We describe next explicitly the Bismut connection $\nabla^b$ on Vaisman manifolds. We begin by computing $\nabla^b_A$. Since $c(A,\cdot,\cdot) = 0$ (see Corollary~\ref{torsion-A}), we have that $\nabla^b_AY = \nabla^g_AY$ for any $Y\in\mathfrak{X}(M)$.  Moreover, due to $\nabla^g A=0$, we have that $\nabla^g_AY=[A,Y]$ and therefore $\nabla^b_AY=[A,Y]$ for all $Y\in\mathfrak{X}(M)$. 

\medskip
	
Next, we determine $\nabla^b_{JA}$. In order to do this, we compute first $g(\nabla^g_{JA}X,Y)$ for any $X,Y\in\mathfrak{X}(M)$, using the Koszul formula:
\begin{align*}
	g(\nabla^g_{JA}X,Y) = & \frac12\left\{JA g(X,Y)+Xg(Y,JA)-Yg(JA,X) \right. \\ 
	& \qquad \left. +g([JA,X],Y)-g([X,Y],JA)+g([Y,JA],X)\right\}.
\end{align*}
Since $JA$ is a Killing vector field, we have that $JA g(X,Y)=g([JA,X],Y)+g(X,[JA,Y])$, so that the expression above becomes
	\begin{align*}
		g(\nabla^g_{JA}X,Y) & =\frac12\left\{2g([JA,X],Y)+X(J\theta(Y))-Y(J\theta(X))-J\theta([X,Y]) \right\}\\
		& = g([JA,X],Y)+\frac12 d(J\theta)(X,Y)\\
		& = g([JA,X],Y)-\frac12 g(\varphi(X),Y)\\
		& = g\left([JA,X]-\frac12\varphi(X),Y\right).
\end{align*}
Hence we obtain $\nabla^g_{JA}X=[JA,X]-\frac12\varphi(X)$ for any $X\in\mathfrak{X}(M)$. Now, using \eqref{bismut} and Corollary~\ref{torsion-JA}:
	\begin{align*}
		g(\nabla^b_{JA}X,Y) = & g(\nabla^g_{JA}X,Y) +\frac12 c(JA,X,Y) \\
		= & g([JA,X]-\frac12\varphi(X),Y)-\frac12 g(\varphi(X),Y)\\
		= & g([JA,X]-\varphi(X),Y),
	\end{align*}
so that $\nabla^b_{JA}X=[JA,X]-\varphi(X)$.
	
	%we use \eqref{nablab-formula}:
	%\begin{align*}
	%	\nabla^b_{JA}X & =\nabla^g_{JA}X+\frac12(-\theta(A)JX+\theta(JX)A-\omega(JA,X)JA)\\
	%	                        & =[JA,X]-\frac12 \varphi(X)+\frac12 (-JX-J\theta(X)A+\theta(X)JA)\\
	%	                        & =[JA,X]-\frac12 \varphi(X)-\frac12 \varphi(X)\\
	%	                        & =[JA,X]-\varphi(X),
	%\end{align*}
	%for any vector field $X$ on $M$.
	
\medskip
	
Finally, we obtain from \eqref{nablab-formula} that, for $X,Y\in\Gamma(\mathcal{D})$,
	\begin{equation}\label{XYenD}
	    \nabla^b_XY=\nabla^g_XY-\frac12 \omega(X,Y)JA,
	\end{equation}
with $\nabla^b_XY\in\Gamma(\mathcal D)$. Indeed, observe that
	\[ g(\nabla^b_XY,A)=Xg(Y,A)-g(Y,\nabla^b_XA)=0 \]
and also
	\[ g(\nabla^b_XY,JA)=Xg(Y,JA)-g(Y,\nabla^b_XJA)=0, \]
so that $\nabla^b_XY\in\Gamma(\mathcal{D})$.

	%{\color{red}
	%Finally, we compute $\nabla^g_X$ for $X\in\Gamma(\mathcal{D})$. Consider $X, Y, Z\in \Gamma(\mathcal{D})$. We have that $d\omega(JX,JY,JZ)=\theta\wedge\omega(JX,JY,JZ)=0$ and therefore 
	%$g(\nabla^b_XY,Z)=g(\nabla^g_XY,Z)$. Since we also have that $g(\nabla^g_XY,A)=0$ (see Corollary~\ref{nabla-gamma}), it follows that 
	%\[ \nabla^g_XY=\alpha(X,Y)JA+\nabla^b_XY \]
	%for certain $\alpha(X,Y)\in C^\infty(M)$. In order to determine $\alpha$ we perform the following computations:
	%\begin{align*}
	% \alpha(X,Y) & = g(\nabla^g_XY,JA) \\
	%             & = \frac12\{ -JAg(X,Y)+g([X,Y],JA)-g([Y,JA],X)+g([JA,X],Y)\} \\
	%             & = \frac12 g([X,Y],JA)
	%\end{align*}
	%since $JA$ is a Killing vector field. On the other hand, we expand the identity $d\omega(A,X,Y)=\theta\wedge \omega(A,X,Y)$. The right-hand side is given simply by
	%\[ \theta\wedge \omega(A,X,Y)=\omega(X,Y),\]
	%while the left-hand side can be computed as follows:
	%\begin{align*}
	% d\omega(A,X,Y) & = A(\omega(X,Y))-\omega([A,X],Y)-\omega([Y,A],X)-\omega([X,Y],A) \\
	%                & = A(g(JX,Y))-g([A,JX],Y)-g(JX,[A,Y])+g([X,Y],JA) \quad (\text{since } \mathcal{L}_A J=0) \\
	%                & = g([X,Y],JA) \quad (\text{since $JA$ is Killing}).
	%\end{align*}
	%Therefore, $\alpha(X,Y)=\frac12 \omega(X,Y)$, and in consequence,
	%\[ \nabla^g_XY=\frac12 \omega(X,Y)JA+\nabla^b_XY \]
	%for any $X,Y\in\Gamma(\mathcal D)$.}
	
\
	
To sum up, we state the following theorem.
	
\begin{theorem}\label{nabla_b}
With notation as above, the Bismut connection $\nabla^b$ on the Vaisman manifold $(M,J,g)$ is given by:
	\begin{itemize}
			\item $\nabla^bA=\nabla^bJA=0$,
			\item $\nabla^b_AX=[A,X]$ for any $X\in\mathfrak{X}(M)$,
			\item $\nabla^b_{JA}X=[JA,X]-\varphi(X)$ for any $X\in\mathfrak{X}(M)$, 
			\item if $X,Y\in\Gamma(\mathcal D)$ then $\nabla^b_XY\in \Gamma(\mathcal D)$ and, moreover, $\nabla^b_XY=\nabla^g_XY-\frac12 \omega(X,Y)JA$.
	\end{itemize}
\end{theorem}

\medskip
	
\subsection{Curvature of $\nabla^b$} 
	
We will use the convention $R^b(X,Y)Z=\nabla^b_X\nabla^b_YZ-\nabla^b_Y\nabla^b_XZ-\nabla^b_{[X,Y]}Z$ for the $(1,3)$-curvature tensor $R^b$ of the Bismut connection. We will denote also by $R^b$ the associated $(0,4)$-curvature tensor: $R^b(X,Y,Z,W)=g(R^b(X,Y)Z,W)$.

\

In the following result we state some symmetries of the Bismut curvature tensor on Vaisman manifolds.

\begin{lemma}\label{curvatura-JJ}
On any Vaisman manifold $(M,J,g)$, the curvature tensor $R^b$ of the Bismut connection satisfies:
		\begin{enumerate}
			\item[(a)] $R^b(X,Y)JZ=JR^b(X,Y)Z$,
			\item[(b)] $R^b(X,Y,Z,W)=R^b(Z,W,X,Y)$,
			\item[(c)] $R^b(JX,JY)=R^b(X,Y)$,
			\item[(d)] $R^b(A,X)=R^b(JA,X)=0$,
		\end{enumerate} 
for any vector fields $X,Y,Z,W$ on $M$.
\end{lemma}
	
\begin{proof}
(a) holds for the Bismut connection on any Hermitian manifold, since $\nabla^bJ=0$.
		
\medskip
		
(b) holds for any metric connection with \textit{parallel} skew-symmetric torsion, according to \cite[Lemma 2.2]{CMS}. Recall that this is the case for the Bismut connection on a Vaisman manifold, due to Corollary \ref{parallel torsion}.
		
\medskip
		
(c) follows from (a) and (b). Indeed, for any vector fields $X,Y,Z,W$ on $M$, we have that
		\begin{align*}
			g(R^b(JX,JY)Z,W) & = R^b(JX,JY,Z,W) \\
			& = R^b(Z,W,JX,JY) \\
			& = g(R^b(Z,W)JX,JY) \\
			& = g(R^b(Z,W)X,Y)\\
			& = g(R^b(X,Y)Z,W).
		\end{align*}
		
\medskip

(d) follows from (b). Indeed, for vector fields $X,U,V$ on $M$ we compute
	\[ g(R^b(A,X)U,V)=g(R^b(U,V)A,X)=0   \]
since $A$ is $\nabla^b$-parallel. The analogous result holds for $JA$ since it is also $\nabla^b$-parallel.
\end{proof}

\
	
Next, we will establish an explicit relation between the Bismut curvature $R^b$ and the Riemannian curvature $R^g$. For this, we will use the following formula from \cite{IP}, which in this case has been simplified since the torsion $3$-form $c$ is $\nabla^b$-parallel:
	\begin{align*}
		R^b(X,Y,Z,U)= & R^g(X,Y,Z,U)+\frac12g(T^b(X,Y),T^b(Z,U))\\ 
		& \qquad +\frac14 g(T^b(X,U),T^b(Y,Z))+\frac14 g(T^b(Y,U),T^b(Z,X)), 
	\end{align*}
for any vector fields $X,Y,Z,U$ on $M$. Using the expression for $T^b$ given in Corollary \ref{torsion-12}, and after lengthy computations, we arrive at:
	\begin{align}
		R^b(X,Y)Z & = R^g(X,Y)Z -\frac14\theta(JY) \theta(JZ)X +\frac14 \theta(JX)\theta(JZ)Y \label{Rb-Rg} \nonumber\\
		%	& \qquad \\
			& \quad +\frac14 g(\varphi(Y),Z)JX-\frac14 g(\varphi(X),Z)JY+\frac12g(\varphi(X),Y)JZ \\
			& \quad +\frac14(-\omega(X,Y) \theta(JZ)+J\theta\wedge\omega(X,Y,Z))A \nonumber\\
		\nonumber	& \quad -\frac14(J\theta\wedge\omega(X,Y,JZ)+\theta\wedge\omega(X,Y,Z))JA \nonumber.
 \end{align}

Observe that, since $R^b(JA,\cdot)=0$, we obtain from \eqref{Rb-Rg} the following expression for $R^g(JA,Y)$:
\begin{equation}\label{ric-JA}
R^g(JA,Y)Z= \frac14 \theta(JZ)Y -\frac14 \theta(Y)\theta(JZ)A+\frac14 \{\theta(JY)\theta(JZ)+g(\varphi(Y),JZ) \} JA,
\end{equation}
where we have used Lemma \ref{torsion} and Corollary \ref{coro-phi}.

\

We study now some properties of the Bismut Ricci curvature $\operatorname{Ric}^b$, defined as usual by $\operatorname{Ric}^b(X,Y)=\trace{(Z\to R^b(Z,X)Y)}$. The next result follows easily from Lemma \ref{curvatura-JJ}:

\begin{corollary}\label{ric-sym}
The Bismut Ricci curvature $\Ric^b$ of a Vaisman manifold satisfies:
\begin{enumerate}
    \item[(a)] $\Ric^b$ is symmetric;
    \item[(b)] $\Ric^b(JX,JY)=\Ric^b(X,Y)$ for any vector fields $X,Y$.
\end{enumerate}
\end{corollary}

We point out that $\Ric^b$ being symmetric is not a surprising fact, since it holds for any metric connection with parallel skew-symmetric torsion. 

\medskip

Now, we are able to obtain an expression for $\Ric^b$, the Bismut Ricci curvature, in terms of the Riemannian Ricci curvature $\Ric^g$. Indeed, let us consider a local orthonormal frame of the form $\{A,JA\}\cup\{e_1,\ldots,e_{2n-2}\}$ where $e_i$ is a local section of $\mathcal D$ for each $i$. Therefore, for any vector fields $Y,Z$ on $M$,
		\begin{align*} 
			\Ric^b(Y,Z) & =g(R^b(A,Y)Z,A)+g(R^b(JA,Y)Z,JA)+\sum_i g( R^b(e_i,Y)Z,e_i) \\
			& = \sum_i g(R^b(e_i,Y)Z,e_i) ,
		\end{align*}	
due to Lemma \ref{curvatura-JJ}(d). 
%Using \eqref{Rb-Rg} we obtain
%\begin{align*}
%\Ric^b(Y,Z) & = \Ric^g(Y,Z) \\ 
%			& \qquad+\sum_i \left( -\frac14 \theta(JY)\theta(JZ)+\frac14 g(JZ,e_i)g(JY,e_i)+\frac12 g(Je_i,Y)g(JZ,e_i) \right)\\
%			& = \Ric^g(Y,Z) - \frac{n-1}{2} \theta(JY)\theta(JZ) - \frac14 g(Y,Z).
%\end{align*}
Using \eqref{Rb-Rg} we obtain
\begin{align*}
\Ric^b(Y,Z) & = \sum_i g(R^g(e_i,Y)Z,e_i) \\ 
			& \qquad+\sum_i \left( -\frac14 \theta(JY)\theta(JZ)-\frac14 g(Je_i,Z)g(JY,e_i)+\frac12 g(Je_i,Y)g(JZ,e_i) \right)\\
			& = \sum_i g(R^g(e_i,Y)Z,e_i) 
			-\frac14\sum_i \left( \theta(JY)\theta(JZ)+  g(JZ,e_i)g(JY,e_i) \right).
\end{align*}

Since $\nabla^g A=0$ we have that $R^g(A,\cdot)=0$, hence
\begin{align*}
    \operatorname{Ric}^g(Y,Z) & =g(R^g(JA,Y,)Z,JA)+\sum_i g(R^g(e_i,Y)Z,e_i)\\
    & =\frac14 g(\varphi(Y),JZ)+\sum_i g(R^g(e_i,Y)Z,e_i)\\
    & = \frac14 (g(Y,Z)-\theta(Y)\theta(Z)-\theta(JY)\theta(JZ))+\sum_i g(R^g(e_i,Y)Z,e_i),
\end{align*} 
where we have used \eqref{ric-JA} in the second equality and the definition of $\varphi$ in the third. Therefore, combining both expressions:
\begin{align}
    \Ric^b(Y,Z) & = \operatorname{Ric}^g(Y,Z)-\frac14 (g(Y,Z)-\theta(Y)\theta(Z)-\theta(JY)\theta(JZ)) \label{ricb-ricg}\nonumber\\
    & -\frac14 \left( (2n-2)\theta(JY)\theta(JZ)+g(Y,Z)-g(JZ,A)g(JY,A)-g(JZ,JA)g(JY,JA)\right)  \\
    & = \operatorname{Ric}^g(Y,Z)-\frac12 g(Y,Z)+\frac12 \theta(Y)\theta(Z)-\frac{n-2}{2}\theta(JY)\theta(JZ). \nonumber
\end{align}

\

As expected, according to Corollary \ref{ric-sym}, $\Ric^b$ is symmetric since the expression above is symmetric in $Y$ and $Z$. It was proved in \cite{IP} that the symmetry of $\Ric^b$ is equivalent to the torsion $3$-form being co-closed, thus we obtain: 

\begin{corollary}
On any Vaisman manifold, the Bismut torsion $3$-form $c$ is co-closed.
\end{corollary}

\

On the other hand, concerning the closedness of the torsion $3$-form $c$, the following result shows that $c$ is never closed in high dimensions.

\begin{proposition}\label{c not closed}
On a Vaisman manifold of dimension $2n\geq 6$, the Bismut torsion $3$-form $c$ is not closed. 
\end{proposition}

\begin{proof}
The $3$-form $c$ is given by $c=-J\theta\wedge \omega$, according to Lemma \ref{torsion}. Therefore $dc$ is given by
\[ dc=-d(J\theta)\wedge \omega +J\theta\wedge d\omega=-(d(J\theta)-J\theta\wedge \theta)\wedge \omega. \]  
So, if $dc=0$ then $\eta:=d(J\theta)-J\theta\wedge \theta=0$, since in dimensions at least $6$ the operator $-\wedge \omega$ is injective on $2$-forms.
However, it follows from Corollary \ref{coro-phi}(d) that 
\[ d(J\theta)(A,JA)=0. \]
On the other hand,
\[(J\theta \wedge \theta)(A,JA)=J\theta(A)\theta(JA)-J\theta(JA)\theta(A)=-1. \]
Hence $\eta(A,JA)=1\neq 0$, a contradiction. As a consequence, $dc\neq 0$.
\end{proof}

\smallskip

\begin{remark}
{\rm 
(i) A Hermitian metric whose associated Bismut torsion $3$-form $c$ is closed is called \textit{pluriclosed} or \textit{strong K\"ahler with torsion (SKT)}. This condition is equivalent to $\partial\overline{\partial} \omega=0$. According to Proposition \ref{c not closed}, a Vaisman metric in dimension $\geq 6$ is never pluriclosed. This result was already known in the compact case, since it was proved in \cite{AI} that on a compact Hermitian manifold of dimension at least 6, the Hermitian metric cannot be LCK and pluriclosed simultaneously, unless the metric is K\"ahler.
    
\smallskip (ii) Notice that according to \cite[Theorem A]{FT}, if $(M,J,g)$ is a $4$-dimensional Vaisman manifold then the Hermitian structure $(J,g)$ is pluriclosed and $\nabla^b$ satisfies the first Bianchi identity. In particular, in real dimension $4$ the torsion $3$-form is harmonic. However, more can be said: $c$ is also $\nabla^g$-parallel, which can be seen from the relation $c=-\ast \theta$, proved in \cite{IP}, which holds for any 4-dimensional LCK manifold. Belgun provided in \cite{Bel} the classification of compact complex surfaces which admit Vaisman metrics: they are properly elliptic surfaces, Kodaira surfaces (either primary or secondary), elliptic Hopf surfaces and Hopf surfaces of class $1$.

\smallskip 

(iii) On Vaisman manifolds of dimension greater than or equal to 6, according to Corollary~\ref{parallel torsion}, Proposition~\ref{c not closed} and \cite[Theorem 3.2]{FT}, the Bismut connection does not satisfy the first Bianchi identity, and therefore it is not K\"ahler-like.  However, due to Lemma~\ref{curvatura-JJ}(c), the Bismut connection satisfies the \emph{type condition} (see for instance \cite{AOUV})}. 
\end{remark}

\
	
\section{Bismut holonomy of Vaisman solvmanifolds}\label{solvmanifold}

In this section we will study the Bismut holonomy of a concrete family of Vaisman manifolds; namely, solvmanifolds equipped with invariant Vaisman structures. We will call them simply Vaisman solvmanifolds. In order to perform this analysis, we will use the results appearing in \cite{AO}.

\

Let $G$ be a Lie group with a left invariant complex structure $J$ and a left invariant metric $g$, i.e. the left translations $L_g:G\to G$ defined by $L_g(h)=gh$ for $h\in G$ are both biholomorphisms and isometries. If $(G,J,g)$ satisfies the LCK condition \eqref{dif}, then $(J,g)$ is called a 
{\em left invariant LCK structure} on the Lie group $G$. In this case, it follows from \eqref{d-theta} that the corresponding Lee form $\theta$ on $G$ is also left invariant. 

We will restrict our study to solvable Lie groups equipped with left invariant Vaisman structures. If the solvable Lie group $G$ is simply connected then any left invariant Vaisman structure on $G$ turns out to be globally conformal to a K\"ahler structure. Therefore we will consider quotients $M_\Gamma:=\Gamma\backslash G$ where $\Gamma$ is a co-compact discrete subgroup of $G$, so that $M_\Gamma$ is a compact manifold such that the canonical projection $G\to M_\Gamma$ is a local diffeomorphism. The compact quotient $M_\Gamma$ is not simply connected (as $\pi_1(M_\Gamma)=\Gamma$) and it inherits a Vaisman structure. The aim of this section is to analyze the holonomy of the Bismut connection on $M_\Gamma$ associated to this induced structure.

A co-compact discrete subgroup $\Gamma$ of a simply connected solvable Lie group $G$ is called a lattice and the quotient $M_\Gamma=\Gamma\backslash G$ is known as a solvmanifold. We point out that, according to \cite{Mi}, if $G$ admits a lattice then $G$ is unimodular (i.e., $\tr \ad_x=0$ for all $x\in \operatorname{Lie}(G)$). 

\ 

Since we are dealing with left invariant structures on Lie groups, we can work at the Lie algebra level. Therefore we will consider LCK or Vaisman structures on Lie algebras, that is, a Hermitian structure $(J,\pint)$ on a Lie algebra $\frg$, where $\pint$ is an inner product on $\frg$ and $J:\frg\to\frg$ is a skew-symmetric endomorphism of $\frg$ that satisfies 
\[ J^2=-\operatorname{I}, \quad \text{and} \quad [Jx,Jy]-[x,y]-J([Jx,y]+[x,Jy])=0,\]    for any $x,y\in \frg$. Moreover, $d\omega=\theta\wedge \omega$ for some closed $1$-form $\theta\in \frg^*$, and $\nabla^g\theta= 0$ in the Vaisman case. 

As before, let $A\in \frg$ denote the vector dual to $\theta$, i.e., $\theta(U)=\langle A,U\rangle$ for all $U\in\frg$. We may assume $|A|=1$. In this context, Proposition \ref{propiedades-Vaisman} takes the following form:

\begin{proposition}%\label{AJA0}
If $(\frg,J,\pint)$ is Vaisman then
\begin{enumerate}
	\item[(a)] $[A,JA]=0$, 
	\item[(b)] $\ad_A$ and $\ad_{JA}$ are skew-symmetric;
	\item[(c)] $J\circ\ad_A=\ad_A\circ J$.
\end{enumerate}
\end{proposition}

\ 

Solvable Lie groups equipped with left invariant Vaisman structures, and their associated Vaisman solvmanifolds,  were studied in \cite{AO}. We will recall some of the results from that article that will be needed for our study.

\medskip

\begin{lemma}\label{lem-centro}\cite{AO}
Let $\frg$ be a unimodular solvable Lie algebra equipped with a Vaisman structure $(J,\pint)$ and let $\mathfrak{z}(\frg)$ denote the center of $\frg$. Then $JA\in\mathfrak{z}(\frg)$. Moreover $\frz(\frg)\subset \operatorname{span}\{A,JA\}$.
\end{lemma}
	
\medskip

The subspace $\ker \theta$ is in fact an ideal of $\frg$, since $\theta$ is closed, and $JA\in\ker\theta$. Denoting $\frk:=(\text{span}\{A,JA\})^\perp$ (which plays the role of $\mathcal{D}$ in Section 3), we have a decomposition 
\[\ker\theta=\R JA \stackrel{\perp}{\oplus}\mathfrak{k}. \]
For $x,y\in\frk$, we have that $[x,y]\in \ker\theta$ and it can be proved that 
\begin{equation}\label{ka}
		[x,y]=\omega(x,y)JA + [x,y]_\frk, %\quad [x,y]_\frk\in\frk,
\end{equation}
where $[x,y]_\frk$ is the component in $\frk$ of $[x,y]$.

\medskip
	
It follows from \cite{AO} that $[\cdot,\cdot]_\mathfrak{k}$ is a Lie bracket on $\frk$ and, moreover,  $(\mathfrak{k},[\cdot,\cdot]_\mathfrak{k},J|_\mathfrak{k},\pint|_\mathfrak{k})$ is a  K\"ahler Lie algebra. Therefore $\ker\theta$ is a $1$-dimensional central extension of $(\frk,[\cdot,\cdot]_\mathfrak{k})$: $\ker\theta=\R JA\oplus_\omega \mathfrak{k}$.
	
Moreover, since $\frg$ is unimodular we have that $\frk$ is unimodular as well. Due to a classical result of Hano \cite{Hano}, it follows that $\pint|_\frk$ is flat. The main result in \cite{AO} is:
	
\begin{theorem}\cite{AO}
If $(\frg,J,\pint)$ is Vaisman with $\frg$ unimodular and solvable, then:
\[ \frg=\R A\ltimes (\R JA\oplus_\omega \frk) , \]
where:
\begin{itemize}
	\item $\ad_A$ is a skew-symmetric derivation of $\ker\theta=\R JA\oplus_\omega \mathfrak{k}$ with $\ad_A(JA)=0$;
	\item $(\mathfrak{k},J_{\mathfrak k}, \pint|_\mathfrak{k})$ is a K\"ahler flat Lie algebra;
	\item $D:=\ad_A|_{\mathfrak k}$ is a skew-symmetric derivation of $(\mathfrak{k},\pint|_\mathfrak{k})$ which commutes with $J|_{\mathfrak k}$ (i.e. $D\in\mathfrak{u}(\mathfrak k)$).
		\end{itemize}
\end{theorem}
	
%Recalling the description of flat Lie algebras given by Milnor in \cite{Mi} (and later refined in \cite{BDF}), it was shown in \cite{AO} that
	
%\begin{proposition}\label{flat}
%There is an orthogonal decomposition
%$\mathfrak{k}=\frz \oplus \mathfrak{h} \oplus \mathfrak{k}'$, where
%\begin{itemize}
%	\item $\frz$ is the center of $\mathfrak{k}$;
%	\item $\mathfrak{k}':=[\mathfrak{k},\mathfrak{k}]_\frk$ and $\mathfrak{h}$ are abelian;
%	\item $\ad^\frk:\mathfrak{h}\to\mathfrak{so(k')}$ is injective and $\mathfrak{k}'$ is even dimensional. In particular, $\dim \mathfrak{h}\leq \frac{\dim \mathfrak{k}'}{2}$;
%	\item $\ad^\frk_x=\nabla^\frk_x$ for any $x\in\mathfrak{z}\oplus\mathfrak{h}$;
%	\item $\nabla^\frk_x=0$ if and only if $x\in\mathfrak{z}\oplus\mathfrak{k}'$.
%\end{itemize}
%Moreover, 
%\begin{itemize}
%	\item $\mathfrak{z}\oplus\mathfrak{h}$ and $\mathfrak{k}'$ are $J$-invariant;
%	\item $\ad^\frk_H\circ J=J\circ\ad^\frk_H$, for any $H\in\frh$;
%	\item $D$ preserves each space $\frz$, $\frh$ and $\frk'$ and, moreover, $D(\frh+ J\frh)=0$.
%\end{itemize}
%\end{proposition}
	
\medskip

\begin{example}
{\rm In \cite{AO} many examples of unimodular solvable Lie algebras were provided. We recall here one such family of examples. Let us consider the Lie algebras $\frg$ with basis $\{ A,B, e_1,\dots,e_{2n-2}\}$ and Lie bracket given by 
\[
    [A,e_{2i-1}]=a_i e_{2i}, \quad [A,e_{2i}]=-a_i e_{2i-1}, \quad  [e_{2i-1},e_{2i}]=B, \quad i=1,\ldots, n-1, 
\]
for some $a_i\in \R$. Let $\pint$ denote the inner product on $\frg$ such that the basis above is orthonormal, and let $J$ denote the skew-symmetric complex structure on $\frg$ given by
\[ JA=B, \quad Je_{2i-1}=e_{2i}, \quad i=1,\ldots, n. \]
Then it is easy to verify that the Hermitian structure $(J,\pint)$ is Vaisman, where the Lee form $\theta$ is the metric dual of $A$: $\theta(\cdot)=\langle A,\cdot \, \rangle$. Note that $\ker \theta=\text{span}\{B, e_1,\dots,e_{2n-2} \}$ is isomorphic to the $(2n-1)$-dimensional Heisenberg Lie algebra $\frh_{2n-1}$ (so that $\frg=\R\ltimes \frh_{2n-1}$), and the subspace $\frk=\text{span}\{ e_1,\dots,e_{2n-2}\}$, equipped with the Lie bracket $[\cdot,\cdot]_\mathfrak{k}$, is an abelian Lie algebra (which is clearly a flat K\"ahler Lie algebra equipped with the restrictions of $(J,\pint)$.

It was also shown in \cite{AO} that whenever $a_i\in\Q$ for every $i$ the corresponding simply connected Lie group admits lattices. If $a_i=0$ for all $i$, then $\frg$ is the direct product $\frg=\R\times \frh_{2n-1}$, with the well-known Vaisman structure given in \cite{CFL}. 
}
\end{example}

\
	
We compute next the Bismut connection on unimodular solvable Lie algebras equipped with Vaisman structures, using Theorem \ref{nabla_b}. We denote by $\nabla^\frk$ the (flat) Levi-Civita connection on the K\"ahler Lie algebra $\frk$. Recall the skew-symmetric operator $\varphi$ defined in \eqref{tensor-phi}; it satisfies $\varphi(A)=\varphi(JA)=0$ and $\varphi(x)= Jx$ for $x\in\frk$.
	
	%\begin{lemma}
	%The Levi-Civita connection $\nabla^g$ of the Vaisman metric $g$ on $\frg$ is given as follows:
	%\begin{itemize}
	%	\item $\nabla^gA=0$, $\nabla^g_A=\ad_A$,
	%	\item $\nabla^g_{JA}A=\nabla^g_{JA}JA=0$, $\nabla^g_{JA}x=-\frac12 Jx$, $x\in\frk$,
	%	\item $\nabla^g_x JA=-\frac12 Jx$ for any $x\in\frk$,
	%	\item $\nabla^g_xy=\frac12 \omega(x,y)JA+\nabla^\frk_xy$ for any $x,y\in\frk$.
	%\end{itemize}
	%\end{lemma}
	%
	%\
	
	%Using this lemma and \eqref{bismut}, we obtain
	
\begin{lemma}\label{bismut-g}
The Bismut connection $\nabla^b$ on $\frg$ is given as follows:
	\begin{itemize}
		\item $\nabla^bA=\nabla^b JA=0$,
		\item $\nabla^b_Ax=[A,x] \in \frk$ for any $x\in\frg$,
		\item $\nabla^b_{JA}x=-\varphi(x)$ for any $x\in\frg$,
		\item $\nabla^b_xy=\nabla^\frk_xy\in \frk$ for any $x,y\in\frk$.
	\end{itemize}
\end{lemma}
	
\begin{proof}
The first three items follow directly from Theorem \ref{nabla_b}, recalling that $JA$ is a central element of $\frg$, due to Lemma \ref{lem-centro}. As for the fourth, we compute $\nabla^g_xy$ for $x,y\in\frk$. Since $\nabla^gA=0$, we have that 
\[ \langle\nabla^g_xy,A\rangle=-\langle y,\nabla^g_x A\rangle=0. \]
On the other hand, we know that $\nabla^b_xy\in \frk$ (Theorem \ref{nabla_b}) and it follows from \eqref{XYenD} that 
\[ \langle\nabla^g_xy,JA\rangle=\frac12 \omega(x,y).\] 
For $z\in\frk$, we have
\begin{align*}
	\langle\nabla^g_xy,z\rangle & = \frac12\{\langle [x,y],z\rangle -\langle [y,z],x\rangle+\langle [z,x],y\rangle\} \\
	& = \frac12\{\langle[x,y]_\frk,z\rangle-\langle [y,z]_\frk,x\rangle+\langle [z,x]_\frk,y\rangle\} \quad \text{(using \eqref{ka})}\\ 
	& = \langle \nabla^\frk_xy,z\rangle.
\end{align*}
Therefore, $\nabla^g_xy=\frac12 \omega(x,y)JA+\nabla^\frk_xy$ for any $x,y\in\frk$. Comparing with Theorem \ref{nabla_b} we obtain $\nabla^b_xy=\nabla^\frk_xy$, $x,y\in\frk$.
\end{proof}
	
\
	
Finally, we are able to compute the curvature tensor $R^b$ of the Bismut connection on $\frg$ in terms of the endomorphism $\varphi$. 

\begin{theorem}\label{curvature}
If $R^b$ denotes the curvature tensor of the Bismut connection, then $R^b$ is given by 
\[ R^b(u,v)=\langle \varphi( u),v\rangle \varphi, \qquad u,v\in\frg.\]
\end{theorem}

\begin{proof}
Note first that $R^b(u,v)A=R^b(u,v)JA=0$, since both $A$ and $JA$ are $\nabla^b$-parallel. Therefore, we only have to compute $R^b(u,v)$ when evaluated in elements of $\frk$. 
		
Next, recall that $R^b(A,\cdot)=R^b(JA,\cdot)=0$, according to Lemma \ref{curvatura-JJ}(d). Thus, we only have to compute $R^b(x,y)z$ for $x,y,z\in\frk$. First note that, according to \eqref{ka}, 
		\[ \nabla^b_{[x,y]}z=\nabla^b_{[x,y]_\frk}z+\omega(x,y)\nabla^b_{JA}z = \nabla^b_{[x,y]_\frk}z -\omega(x,y)Jz=\nabla^\frk_{[x,y]_\frk}z -\omega(x,y)Jz, \]
where we have used Lemma \ref{bismut-g} in the last equality.	Hence we have that
		\begin{align*}
			R^b(x,y)z & = \nabla^b_x\nabla^b_y z-\nabla^b_y\nabla^b_x z - \nabla^b_{[x,y]}z    \\
			& = \nabla^\frk_x\nabla^\frk_y z-\nabla^\frk_y\nabla^\frk_x z-\nabla^\frk_{[x,y]_\frk}z +\omega(x,y)Jz \\
			& = R^\frk(x,y)z+\langle Jx,y\rangle Jz \\
			& = \langle Jx,y\rangle Jz
		\end{align*}
		since $\nabla^\frk$ is flat. The result follows.
	\end{proof}

	\medskip
	
\begin{corollary}\label{endo}
Any curvature endomorphism $R^b(u,v)$, $u,v\in\frg$, is parallel with respect to~$\nabla^b$.
\end{corollary}
	
\begin{proof}
This is a straightforward consequence of Theorem \ref{curvature} and Corollary \ref{phi paralelo}.
\end{proof}
	
\
	
Regarding the holonomy group of the Bismut connection of a Vaisman solvmanifold, we have
	
\begin{theorem}\label{phi}
If $M=\Gamma\backslash G$ is a $2n$-dimensional Vaisman solvmanifold  then its holonomy group $\operatorname{Hol}^b(M)$ has dimension $1$ and it is not contained in $\operatorname{SU}(n)$.
\end{theorem}
	
\begin{proof}
The restricted holonomy group $\operatorname{Hol}^b_0(M)$ coincides with the holonomy group $\operatorname{Hol}^b(G)$. According to Theorem \ref{AS2}, its Lie algebra $\mathfrak{hol}^b(M)$ is generated by all the curvature endomorphisms $R^b(u,v)$, $u,v\in\frg$, together with their covariant derivatives of any order. Therefore, it follows from Theorem \ref{curvature} and Corollary \ref{endo} that $\mathfrak{hol}^b(M)$ is spanned by $\varphi$, therefore it is one-dimensional. 
		
Moreover, in an adapted basis $\{A,JA,e_1,f_1,\ldots,e_{n-1},f_{n-1}\}$ with $Je_i=f_i$, we have that the matrix of $\varphi$ is given by
		\[ \varphi=\left(\begin{array}{ccccccc}      
			0 & 0 &  &    &              &      &     \\
			0 & 0 &  &    &              &      &     \\
			%  &   & 0_{l\times l} &    &      &        &      &     \\
			&   &  0  & -1   &     &      &     \\
			&   &  1 & 0     &       &      &     \\
			&   &  &    &       \ddots &      &     \\
			&   &  &    &              &   0  & -1\\
			&   &  &    &              &  1 &  0  \\
		\end{array}\right)\in\mathfrak{u}(n),\]
		but it is clear that $\varphi\notin \mathfrak{su}(n)$. 
	\end{proof}
	
%\footnote{AA-3julio: >decir que $Hol_0^b\cong S^1$? No sé qué pasa con $Hol^b$.}
	
\ 

Moreover, a result stronger than Corollary \ref{endo} can be obtained also as a consequence of Theorem \ref{curvature}:
 
\begin{proposition}
On any Vaisman solvmanifold $\Gamma\backslash G$, the Bismut curvature tensor $R^b$ is $\nabla^b$-parallel: $\nabla^bR^b=0$.
\end{proposition}

\begin{proof}
This is an immediate consequence of Theorem \ref{curvature}. Indeed, we need only verify that $(\nabla^b_xR^b)(y,z)w=0$ for any $x,y,z,w\in\frg$. We compute
\begin{align*}
    (\nabla^b_xR^b)(y,z)w & = \nabla^b_x(R^b(y,z)w)-R^b(\nabla^b_xy,z)w-R^b(y,\nabla^b_xz)w-R^b(y,z)(\nabla^b_xw)\\
    & = \langle\varphi(y),z\rangle \nabla^b_x\varphi(w)-\langle \varphi \nabla^b_xy,z\rangle \varphi(w) - \langle \varphi(y),\nabla^b_xz\rangle \varphi(w) - \langle \varphi(y),z\rangle \varphi \nabla^b_xw.
\end{align*}
The first and the last terms cancel out since $\varphi$ is $\nabla^b$-parallel, and the second and third terms also cancel out, since 
\[ \langle \varphi \nabla^b_xy,z\rangle=\langle \nabla^b_x\varphi(y),z\rangle=-\langle \varphi(y),\nabla^b_x z \rangle. \] 
This completes the proof.
\end{proof}

\begin{remark} 
{\rm According to \cite{AP}, the Bismut connection on a Vaisman solvmanifold is a \textit{Hermitian Ambrose-Singer connection}, since $\nabla^bT^b=0$ and $\nabla^bR^b=0$. In particular, any Vaisman solvmanifold is a locally homogeneous Hermitian space \cite{Ki,Se}. However, this is true for any solvmanifold $M:=\Gamma\backslash G$ equipped with an invariant almost Hermitian structure $(J,g)$, since the connection on $G$ defined by $\nabla_x y=0$ for any $x,y\in\frg=\operatorname{Lie}(G)$ induces a connection $\nabla$ on $M$ satisfying $\nabla J=\nabla g=0$.
}
\end{remark}

\medskip

Concerning the Bismut Ricci curvature of a Vaisman solvmanifold, we have the following straightforward consequence of Theorem \ref{curvature}:

\begin{corollary}
The Bismut Ricci curvature of a Vaisman solvmanifold $\Gamma\backslash G$ is given by
\[ 
\operatorname{Ric}^b(u,v)=-\langle u,v\rangle+\theta(u)\theta(v)+\theta(Ju)\theta(Jv), \quad u,v\in\frg. \]
In particular, $\operatorname{Ric}^b\neq 0$.
\end{corollary}

\smallskip

Using \eqref{ricb-ricg} we are able to determine the Riemannian Ricci curvature of a $2n$-dimensional Vaisman solvmanifold $\Gamma\backslash G$:
\[ \operatorname{Ric}^g(u,v)=-\frac12 \langle u,v\rangle+\frac12 \theta(u)\theta(v)+\frac{n}{2}\theta(Ju)\theta(Jv), \quad u,v\in\frg.\]

\

For a general non-Vaisman LCK solvmanifold, we cannot expect a reduction of the holonomy of the Bismut connection, as the following example shows. 

\begin{example}\label{example}
{\rm %We will determine next the holonomy Lie algebra of a $4$-dimensional non-Vaisman LCK solvmanifold. 
Let $G$ be the simply connected solvable Lie group with Lie algebra $\frg$ generated by $\{e_1,e_2,e_3,e_4\}$ with non-zero brackets given by
	\[ [e_1,e_2]=\mu e_2, \quad  [e_1,e_3]=-\frac{\mu}{2}e_3+y e_4, \quad [e_1,e_4]=-ye_3-\frac{\mu}{2}e_4,   \]
for some $\mu\neq 0$ and $y\in\R$. Note that $G$ is an almost abelian Lie group; it was proved in \cite{AO1} that for certain values of $\mu$ and $y$ the Lie group $G$ admits lattices. The associated solvmanifolds are Inoue surfaces of type $S^0$.
			
Consider on $\frg$ the inner product $\pint$ such that the basis above is orthonormal and the endomorphism $J:\frg\to\frg$ given by $Je_1=e_2,\, Je_3=e_4,\, J^2=-\operatorname{Id}$. It is easy to verify that the almost complex structure $J$ is integrable and hence $(J,\pint)$ determines a Hermitian structure on $\frg$  with associated fundamental $2$-form $\omega$ given by $\omega=e^{12}+e^{34}$. Here, $\{e^1, e^2, e^3, e^4\}$ is the dual basis of $\{e_1, e_2, e_3, e_4\}$ and $e^{ij}$ stands for the wedge product $e^i\wedge e^j$. Note that $d\omega=\mu e^1\wedge \omega$, which means that $(J,\pint)$ is LCK since $\mu\neq 0$ and $de^1=0$. Clearly, the Lee form $\theta$ is $\theta=\mu e^1$. 
			
Computing the Bismut connection on $\frg$ using \eqref{bismut}, we obtain
	\begin{gather*}
		\nabla^b_{e_1}= \begin{pmatrix} \\ \\ && 0& -y \\ && y & 0 \end{pmatrix}, \qquad \nabla^b_{e_2}= \begin{pmatrix} 0 &\mu && \\ -\mu & 0 &&\\ && 0& \frac{\mu}{2} \\ && -\frac{\mu}{2} & 0 \end{pmatrix},  \\ 
		\nabla^b_{e_3}= \begin{pmatrix} && -\frac{\mu}{2} & 0 \\ &&0&-\frac{\mu}{2}\\ \frac{\mu}{2}&0 & & \\ 0 & \frac{\mu}{2} &  &  \end{pmatrix} 
		\qquad \nabla^b_{e_4}= \begin{pmatrix}  &&0&-\frac{\mu}{2}\\  &&\frac{\mu}{2} &0  \\ 0& -\frac{\mu}{2} & & \\ \frac{\mu}{2}&0 & & \end{pmatrix}.
    \end{gather*}
Note that $\nabla^b \theta\neq 0$, so that the LCK metric is not Vaisman. The curvature endomorphisms $R^b(e_i,e_j)$ are given by
\begin{gather*}
		R^b(e_1,e_2)= \begin{pmatrix} 0 & -\mu^2 &&\\ \mu^2 & 0 && \\ && 0 & -\frac{\mu^2}{2} \\ && \frac{\mu^2}{2} & 0 \end{pmatrix}, \quad 
		R^b(e_1,e_3)=-R^b(e_2,e_4)= \begin{pmatrix} && -\frac{\mu^2}{4} & 0 \\ && 0 &-\frac{\mu^2}{4}\\ \frac{\mu^2}{4}& 0 & & \\ 0 & \frac{\mu^2}{4} &  &  \end{pmatrix}   \\ 
		R^b(e_1,e_4)=R^b(e_2,e_3)=
		\begin{pmatrix}  && 0 &-\frac{\mu^2}{4}\\  &&\frac{\mu^2}{4} & 0 \\ 0 & -\frac{\mu^2}{4} & & \\ \frac{\mu^2}{4}& 0 & & \end{pmatrix},
		\quad R^b(e_3,e_4)=\begin{pmatrix}  0 &\frac{\mu^2}{2} && \\ -\frac{\mu^2}{2} & 0 &&\\ && 0 & -\frac{\mu^2}{2} \\ && \frac{\mu^2}{2} & 0 \end{pmatrix}.
    \end{gather*}
Since $\mu\neq 0$, it follows easily that all these curvature endomorphisms are linearly independent, hence the subspace $\text{span}\{R^b(e_i,e_j)\mid i<j \}$ of $\mathfrak{u}(2)$ has dimension $4$. According to Theorem \ref{AS2}, we have that $\mathfrak{hol}^b=\mathfrak{u}(2)$. There is no reduction of the holonomy in this case.

\medskip

This example will be generalized in Section \ref{section-OT}. 
}
\end{example}
	
\
	
\section{Bismut holonomy of Hopf manifolds}\label{Hopf}

In this section we determine explicitly the holonomy group of the Bismut connection on some classical \textit{Hopf manifolds}, which are the archetypical examples of compact Vaisman manifolds. These are defined as a quotient of $\C^n-\{0\}$, $n\geq 2$, by the action of the cyclic group generated by the transformation $z\mapsto \lambda z$, for some $\lambda \in\C$ with $|\lambda|>1$. Their underlying smooth manifold is $S^1\times S^{2n-1}$, so that its first Betti number is $b_1=1$ and therefore they cannot admit a K\"ahler metric. We also point out that for different values of $\lambda$ the corresponding compact complex manifolds are non-biholomorphic (this can be seen with the same arguments used in \cite{KS} for the case of Hopf surfaces, i.e., for $n=2$).

\smallskip

We will describe in more detail their construction for $\lambda\in\R$, $\lambda>1$, exhibiting in this case an explicit parallelization of $S^1\times S^{2n-1}$, generalizing the one given in \cite{Par1} (and in greater length in \cite{Par2}) for $\lambda=e^{2\pi}$. It will be easy to express the usual Vaisman structure in terms of this parallelization, and using this expression we will be able to show that the associated Bismut holonomy group is equal to $\operatorname{U}(n-1)$, which is the largest possible holonomy group, according to Corollary~\ref{hol}.

\smallskip

\begin{remark}
{\rm These classical Hopf manifolds are not homeomorphic to a solvmanifold for $n>1$. Indeed, the universal cover of a $2n$-dimensional solvmanifold is homeomorphic to $\R^{2n}$, whereas the universal cover of $S^1\times S^{2n-1}$ is $\R^{2n}-\{0\}$. } 
\end{remark}

\

\subsection{Revisiting the construction of a Vaisman structure on Hopf manifolds}
Let us consider $\R^{2n}$ with the usual Cartesian coordinates $(x_1,\ldots,x_{2n})$. Let us denote by $N$ the unit normal vector field of $S^{2n-1}\subset \R^{2n}$, that is, $N=\sum_{i=1}^{2n} x_i\frac{\partial}{\partial x_i}$. For any $i=1,\ldots,2n$, the orthogonal projection of $\frac{\partial}{\partial x_i}$ on $S^{2n-1}$ gives a vector field $T_i$ on $S^{2n-1}$, which can be expressed as $T_i=\frac{\partial}{\partial x_i}-x_i N$. The vector field $T_i$ is called the $i^{th}$ meridian vector field on $S^{2n-1}$.
	
Let $\lambda$ be a real number, $\lambda>1$. Let $\Gamma_\lambda$ be the cyclic infinite group of transformations of $\R^{2n}-\{0\}$ generated by the map $x \mapsto \lambda x$. Then the projection $p_\lambda:\R^{2n}-\{0\} \to S^1\times S^{2n-1}$ given by \[ p_\lambda(x)=\left(\exp\left(2\pi i\,\frac{\log|x|}{\log \lambda}\right),\frac{x}{|x|}\right) \]
induces a diffeomorphism between $(\R^{2n}-\{0\})/\Gamma_\lambda$ and $S^1\times S^{2n-1}$. Consider the smooth function $F:\R^{2n}-\{0\}\to \R$ given by $F(x)=|x|$. Then the vector fields $F\frac{\partial}{\partial x_i}$, $i=1,\ldots,2n$, are $\Gamma_\lambda$-invariant\footnote{Here we are using $\lambda >0$, this does not hold for general $\lambda\in \C$.} and therefore the vector fields $U_i^\lambda$ on $S^1\times S^{2n-1}$ given by $U_i^\lambda=(p_\lambda)_\ast(F\frac{\partial}{\partial x_i})$ are well-defined and, moreover, $\{U_1^\lambda,\ldots,U_{2n}^\lambda\}$ defines a parallelization of $S^1\times S^{2n-1}$. In terms of the meridian vector fields $T_i$, it can be shown, modifying suitably the computations in \cite{Par2}, that $U_i^\lambda=T_i+\frac{2\pi}{\log \lambda}x_i\frac{\partial}{\partial t}$, where $t$ denotes the usual coordinate on $S^1$. It follows that 
\begin{gather*}
[U_i^\lambda,U_j^\lambda]= x_iU_j^\lambda-x_jU_i^\lambda, \\
U_i^\lambda(x_j)=\delta_{ij}-x_ix_j,
\end{gather*}
for $i,j=1,\ldots,2n$. Here we are considering the functions $x_j\in C^\infty(S^1\times S^{2n-1})$ defined as $x_j(e^{it},(p_1,\ldots, p_{2n}))=p_j$, for $j=1,\ldots, 2n$ and $(p_1,\ldots,p_{2n})\in S^{2n-1}$. In particular, $\sum_{j=1}^{2n} x_j^2=1$. Observe that differentiating this expression we get
\begin{equation}\label{dxj}
    \sum_{j=1}^{2n} x_j dx_j=0.
\end{equation}
	
If $\pint$ denotes the usual metric on $\R^{2n}$, then the Riemannian metric $\frac{1}{F^2}\langle \cdot,\cdot\rangle$ on $\R^{2n}-\{0\}$ is $\Gamma_\lambda$-invariant and then it induces a Riemannian metric $g_\lambda$ on $S^1\times S^{2n-1}$. It can be easily seen that $g_\lambda$ coincides with the product metric $g_\lambda=\left(\frac{\log\lambda}{2\pi}\right)^2g_{S^1}+g_{S^{2n-1}}$, where $g_{S^k}$ denotes the round metric on $S^k$.
	
On $\R^{2n}-\{0\}$ there is the canonical complex structure given by
\[  J\left(\frac{\partial}{\partial x_{2i-1}}\right) = \frac{\partial}{\partial x_{2i}},\quad J\left(\frac{\partial}{\partial x_{2i}}\right) = -\frac{\partial}{\partial x_{2i-1}}, \quad i=1,\ldots,n. \]
%	\[ \frac{\partial}{\partial x_{2i-1}}\mapsto \frac{\partial}{\partial x_{2i}}, \quad \frac{\partial}{\partial x_{2i}}\mapsto -\frac{\partial}{\partial x_{2i-1}},  \]
This complex structure is $\Gamma_\lambda$-invariant and therefore defines a complex structure $J_\lambda$ on $S^1\times S^{2n-1}$ given by
%\begin{equation}\label{J}
\[	J_\lambda U_{2i-1}^\lambda=U_{2i}^\lambda, \quad J_\lambda U_{2i}^\lambda=-U_{2i-1}^\lambda,\quad i=1,\ldots,n. \]
%\end{equation}
It is clear that $J_\lambda$ is $g_\lambda$-orthogonal and therefore $(J_\lambda,g_\lambda)$ is a Hermitian structure on $S^1\times S^{2n-1}$ for any $\lambda>1$.
	
Let $\eta_i^\lambda$ denote the $1$-form on $S^1\times S^{2n-1}$ which is $g_{\lambda}$-dual to $U_i^\lambda$. Then $\eta_i^\lambda=dx_i+\frac{\log\lambda}{2\pi}x_i dt$ and it can be seen that $d\eta_i^\lambda=\eta_i^\lambda\wedge\alpha^\lambda$, where $\alpha^\lambda$ is the $1$-form defined by $\alpha^\lambda:=\sum_j x_j\eta_j^\lambda$. Note that applying \eqref{dxj} we obtain  $\alpha^\lambda=\frac{\log\lambda}{2\pi}dt$. 

%{\color{blue} Sale de acá: \[ \alpha^\lambda:=\sum_j x_j\eta_j^\lambda=\sum_j x_j(dx_j+\frac{\log\lambda}{2\pi}x_j dt) = \sum_j x_j dx_j+ \frac{\log\lambda}{2\pi} (\sum_j x_j^2) dt.\]
%Usando que $\sum_j x_j^2=1$ sale, derivando, que $\sum_j x_j dx_j=0$, y reemplazando estas dos cosas en la fórmula azul queda la expresión de $\alpha^\lambda$ en rojo. >Te parece que hace falta agregar esta cuenta? Aunque me parece que nunca se usa explícitamente.}

\

We summarize all the equations we have obtained so far in the following lemma: 
%\begin{align*}
%	[U_i^\lambda,U_j^\lambda] &= x_iU_j^\lambda-x_jU_i^\lambda, \\%\label{UU}\\
%	U_i^\lambda(x_j) & =\delta_{ij}-x_ix_j,\\%\label{Ux}
%	\{U_1,\ldots,U_{2n}\} & \text{ is a $g_\lambda$-orthonormal frame on } S^1\times S^{2n-1},\\
%	d\eta_i^\lambda & =\eta_i^\lambda\wedge \alpha^\lambda,\\
%	J_\lambda U_{2i-1}& =U_{2i}, \; JU_{2i}=-U_{2i-1} \text{ for all } i.\\ 
%\end{align*}
	
\begin{lemma}\label{lema-hopf}
The manifold $S^1\times S^{2n-1}$ admits a family of Hermitian structures $(J_{\lambda}, g_{\lambda})$ for $\lambda>1$.  In terms of the $g_\lambda$-orthonormal global frame $\{U_1^\lambda,\ldots,U_{2n}^\lambda\}$ and functions $x_i\in C^\infty(S^1\times S^{2n-1})$, $i=1,\ldots,2n$, described above, we have:
\begin{align*}
J_\lambda U_{2i-1}^\lambda& =U_{2i}^\lambda, \quad  J_\lambda U_{2i}^\lambda=-U_{2i-1}^\lambda \text{ for all } i,\\
[U_i^\lambda,U_j^\lambda] &= x_iU_j^\lambda-x_jU_i^\lambda, \quad
U_i^\lambda(x_j)  =\delta_{ij}-x_ix_j,
\end{align*}
and 
\[ d\eta_i^\lambda  =\eta_i^\lambda\wedge \alpha^\lambda,\] 
where  $\eta_i^{\lambda}$ is the form $g_{\lambda}$-dual to $U_i^{\lambda}$ and $\alpha^\lambda:=\sum_j x_j\eta_j^\lambda$.
\end{lemma}

\medskip

Note that these expressions do not actually depend on $\lambda$, so that from now on we will omit the subscript/superscript $\lambda$ in all the forthcoming computations. 

\
	
We will verify next the well-known fact that the previous Hermitian structure $(J,g)$ is Vaisman. Indeed, computing $d\omega$ for $\omega:=g(J\cdot,\cdot)=\sum_i \eta_{2i-1}\wedge\eta_{2i}$ and applying \eqref{dxj}, we obtain that 
\[ d\omega=-2\alpha\wedge\omega,\quad d\alpha=0,\]
so that $(J,g)$ is LCK with Lee form $\theta=-2\alpha$. The vector field $H$ on $S^1\times S^{2n-1}$ which is $g$-dual to $\alpha$ will play an important role in our computations. It can be written in terms of the frame $\{U_i\}$ as
\[ H=\sum_{i=1}^{2n} x_i U_i,\]  
and coincides with $\frac{2\pi}{\log \lambda}\frac{\partial}{\partial t}$. Observe that $H$ is a multiple of the vector field $A$ defined in previous sections for any Vaisman manifold; more precisely, $H = -A/2$. In order to show that $\alpha$, or equivalently $H$, is $\nabla^g$-parallel, we determine the Levi-Civita connection $\nabla ^g$ of $g$ in terms of $\{U_i\}$. Using the Koszul formula together with Lemma \ref{lema-hopf} we obtain
	
\begin{lemma}\label{nabla_g}
The Levi-Civita connection $\nabla^g$ on $S^1\times S^{2n-1}$ is given by
\begin{itemize}
	\item $\nabla^g_{U_i}U_j=-x_j U_i$, if $i\neq j$;
	\item $\nabla^g_{U_i}U_i= \sum_{k\neq i}x_kU_k=H-x_iU_i$.
\end{itemize}
\end{lemma}
	
\smallskip
	
Therefore,
\begin{align*}
	\nabla^g_{U_i}H & =\nabla^g_{U_i}(\sum_j x_jU_j) \\
	& = \sum_j(U_i(x_j)U_j+x_j\nabla^g_{U_i}U_j) \\
	& = (1-x_i^2)U_i+x_i(H-x_iU_i)+\sum_{j\neq i} (-x_ix_jU_j+x_j(-x_j)U_i)\\
	& = (1-x_i^2)U_i+x_i(H-x_iU_i)-x_i(H-x_iU_i)-(\sum_{j\neq i}x_j^2) U_i\\
	&= 0 \qquad (\text{since} \sum_j x_j^2=1).
\end{align*}
Thus we recover the fact that $S^1\times S^{2n-1}$ with the Hermitian structure $(J,g)$ is a Vaisman manifold.

\

\subsection{Computation of the Bismut holonomy of Hopf manifolds}
In what follows, we will study the Bismut connection $\nabla^b$ associated to $(J,g)$. We will first express $\nabla^b$ in terms of the frame $\{U_i\}$, and later we will determine its curvature tensor $R^b$ and its holonomy group. 
	
\begin{proposition}
The Bismut connection $\nabla^b$ associated to $(J,g)$ on $S^1\times S^{2n-1}$ is given by
\begin{itemize}
	\item $\nabla^b_{U_{2i-1}}U_{2j-1}=-x_{2j-1}U_{2i-1}+x_{2j}U_{2i}-x_{2i}U_{2j}$, $(i\neq j)$;
	\item $\nabla^b_{U_{2i-1}}U_{2i-1}=\sum_{k\neq 2i-1}x_kU_k=H-x_{2i-1}U_{2i-1}$;
	\item $\nabla^b_{U_{2i-1}}U_{2j}= -x_{2j}U_{2i-1}+x_{2i}U_{2j-1}-x_{2j-1}U_{2i}$, $(i\neq j)$;
	\item $\nabla^b_{U_{2i-1}}U_{2i}=-\sum_{k}x_{2k}U_{2k-1}+\sum_{k\neq i}x_{2k-1}U_{2k}=JH-x_{2i-1}U_{2i}$;
	\item $\nabla^b_{U_{2i}}U_{2j-1}=-x_{2j}U_{2i-1}-x_{2j-1}U_{2i}+x_{2i-1}U_{2j}$, $(i\neq j)$;
	\item $\nabla^b_{U_{2i}}U_{2i-1}=-\sum_{k}x_{2k-1}U_{2k}+\sum_{k\neq i}x_{2k}U_{2k-1}=-JH-x_{2i}U_{2i-1}$;
	\item $\nabla^b_{U_{2i}}U_{2j}=-x_{2i-1}U_{2j-1}+x_{2j-1}U_{2i-1}-x_{2j}U_{2i}$, $(i\neq j)$;
	\item $\nabla^b_{U_{2i}}U_{2i}=\sum_{k\neq 2i}x_{k}U_{k}=H-x_{2i}U_{2i}$.
\end{itemize}
\end{proposition}
	
\begin{proof}
The proof follows from \eqref{bismut}, using Lemma \ref{nabla_g} and the fact that $d\omega=-2\alpha\wedge \omega$, where $\alpha=\sum x_j\eta_j, \, \omega=\sum \eta_{2j-1}\wedge \eta_{2j}$. We prove only the first two expressions, the others follow analogously. Also take into account that $\nabla^b_X U_{2r} = J(\nabla^b_X U_{2r-1})$ since $J$ is $\nabla^b$-parallel.
		
\medskip
		
For $i\neq j$:
\begin{align*}
	g(\nabla^b_{U_{2i-1}}U_{2j-1},U_{2k-1}) & = -x_{2j-1}\delta_{ik}-\alpha\wedge \omega(U_{2i},U_{2j},U_{2k})\\
	& = -x_{2j-1}\delta_{ik},
	\end{align*}
while
\begin{align*}
	g(\nabla^b_{U_{2i-1}}U_{2j-1},U_{2k}) & = \alpha\wedge \omega(U_{2i},U_{2j},U_{2k-1})\\
	& = \alpha(U_{2i})\omega(U_{2j},U_{2k-1})+\alpha(U_{2j})\omega(U_{2k-1},U_{2i})\\
	& = -x_{2i}\delta_{jk}+x_{2j}\delta_{ik}.
\end{align*}
Therefore: $\nabla^b_{U_{2i-1}}U_{2j-1}=-x_{2j-1}U_{2i-1}+x_{2j}U_{2i}-x_{2i}U_{2j}$. Now, for $i=j$, and for any vector field $X$ on $S^1\times S^{2n-1}$ we have that 
\[ g(\nabla^b_{U_{2i-1}}U_{2i-1},X)=g(\nabla^g_{U_{2i-1}}U_{2i-1},X)-\alpha\wedge\omega(U_{2i},U_{2i},JX)=g(\nabla^g_{U_{2i-1}}U_{2i-1},X).  \]
Therefore, $\nabla^b_{U_{2i-1}}U_{2i-1}=\nabla^g_{U_{2i-1}}U_{2i-1}=H-x_{2i-1}U_{2i-1}$, as claimed.
\end{proof}
	
\
	
Next, we will determine the curvature endomorphisms $R^b(U_i,U_j)$ associated to $\nabla^b$. This will be done in Propositions \ref{curvatura1}, \ref{curvatura3} and \ref{curvatura2}; their proofs are long but straightforward computations and therefore we omit them\footnote{Another expression for the Bismut curvature tensor on Hopf manifolds was given recently in \cite{B}.}.

\begin{proposition}\label{curvatura1}
For any $i=1,\ldots,n$, the curvature endomorphism $R^b(U_{2i-1},U_{2i})$ is given by  
\begin{align*}
	R^b(U_{2i-1},U_{2i}) U_{2i-1} & = R^b(U_{2i-1},U_{2i})U_{2i} = 0, \\
	R^b(U_{2i-1},U_{2i}) U_{2j-1} & = 2(1-x_{2i-1}^2-x_{2i}^2-x_{2j-1}^2-x_{2j}^2)U_{2j} \\
	& \quad +2\sum_{k\neq i,j}(x_{2j-1}x_{2k}-x_{2j}x_{2k-1})U_{2k-1} -2\sum_{k\neq i,j}(x_{2j-1}x_{2k-1}+x_{2j}x_{2k})U_{2k}
	%& \quad =2\left\{(1-x_{2i-1}^2-x_{2i}^2-x_{2j-1}^2-x_{2j}^2)U_{2j}-x_{2j}A-x_{2j-1}JA\right\} \\
	% R^b(U_{2i-1},U_{2i}) U_{2j} & = -2(1-x_{2i-1}^2-x_{2i}^2-x_{2j-1}^2-x_{2j}^2)U_{2j-1} \\
	% & \quad +2 \sum_{k\neq i,j}(x_{2j-1}x_{2k-1}+x_{2j}x_{2k})U_{2k-1}+2 \sum_{k\neq i,j}(x_{2j-1}x_{2k}-x_{2j}x_{2k-1})U_{2k}.
	% %& \quad =2\left\{-(1-x_{2i-1}^2-x_{2i}^2-x_{2j-1}^2-x_{2j}^2)U_{2j}+x_{2j-1}A-x_{2j}JA\right\}.
\end{align*} and $R^b(U_{2i-1},U_{2i})U_{2j}=J(R^b(U_{2i-1},U_{2i})U_{2j-1})$ by Lemma~\ref{curvatura-JJ}(a).
\end{proposition}

\
	
\begin{proposition}\label{curvatura3}
For any $i,j=1,\ldots,n$, $(i\neq j)$, the curvature endomorphism $R^b(U_{2i-1},U_{2j})$ is given by 
\begin{align*}
	R^b(U_{2i-1},U_{2j})U_{2i-1} & = -(1-x_{2i-1}^2-x_{2i}^2-x_{2j-1}^2-x_{2j}^2)U_{2j} \\
	& \qquad +\sum_{k\neq i,j}(x_{2j}x_{2k-1}-x_{2j-1}x_{2k})U_{2k-1}+ \sum_{k\neq i,j}(x_{2j-1}x_{2k-1}+x_{2j}x_{2k})U_{2k} \\
	%  R^b(U_{2i-1},U_{2j})U_{2i} & = (1-x_{2i-1}^2-x_{2i}^2-x_{2j-1}^2-x_{2j}^2)U_{2j} \\
	%  & \qquad -\sum_{k\neq i,j}(x_{2j-1}x_{2k-1}+x_{2j}x_{2k})U_{2k-1}+ \sum_{k\neq i,j}(x_{2j}x_{2k-1}-x_{2j-1}x_{2k})U_{2k} \\
	R^b(U_{2i-1},U_{2j})U_{2j-1} & = -(1-x_{2i-1}^2-x_{2i}^2-x_{2j-1}^2-x_{2j}^2)U_{2i} \\
	& \qquad +\sum_{k\neq i,j}(x_{2i}x_{2k-1}-x_{2i-1}x_{2k})U_{2k-1}+ \sum_{k\neq i,j}(x_{2i-1}x_{2k-1}+x_{2i}x_{2k})U_{2k} \\
	%  R^b(U_{2i-1},U_{2j})U_{2j} & = (1-x_{2i-1}^2-x_{2i}^2-x_{2j-1}^2-x_{2j}^2)U_{2i-1} \\
	%  & \qquad -\sum_{k\neq i,j}(x_{2i-1}x_{2k-1}+x_{2i}x_{2k})U_{2k-1} + \sum_{k\neq i,j}(x_{2i}x_{2k-1}-x_{2i-1}x_{2k})U_{2k}\\
	R^b(U_{2i-1},U_{2j})U_{2k-1} & = (x_{2j-1}x_{2k}-x_{2j}x_{2k-1})U_{2i-1}+(x_{2j-1}x_{2k-1}+x_{2j}x_{2k})U_{2i}  \\
	& \qquad +(x_{2i-1}x_{2k}-x_{2i}x_{2k-1})U_{2j-1}+(x_{2i-1}x_{2k-1}+x_{2i}x_{2k})U_{2j} \\
	& \qquad - 2(x_{2i-1}x_{2j-1}+x_{2i}x_{2j})U_{2k} \quad (k\neq i,j)\\
	%  R^b(U_{2i-1},U_{2j})U_{2k} & = -(x_{2j-1}x_{2k-1}+x_{2j}x_{2k})U_{2i-1}+(x_{2j-1}x_{2k}-x_{2j}x_{2k-1})U_{2i} \\
	%  & \qquad -(x_{2i-1}x_{2k-1}+x_{2i}x_{2k})U_{2j-1}+(x_{2i-1}x_{2k}-x_{2i}x_{2k-1})U_{2j} \\
	%  & \qquad + 2(x_{2i-1}x_{2j-1}+x_{2i}x_{2j})U_{2k-1} \quad (k\neq i,j).
\end{align*}
Moreover, $R^b(U_{2i-1},U_{2j})U_{2r}=J(R^b(U_{2i-1},U_{2j})U_{2r-1})$ by Lemma~\ref{curvatura-JJ}(a). 
\end{proposition}
	
\
	
\begin{proposition}\label{curvatura2}
For any $i,j=1,\ldots,n$, $(i\neq j)$, the curvature endomorphism $R^b(U_{2i-1},U_{2j-1})$ is given by  
\begin{align*}
	R^b(U_{2i-1},U_{2j-1})U_{2i-1} & = -(1-x_{2i-1}^2-x_{2i}^2-x_{2j-1}^2-x_{2j}^2)U_{2j-1} \\
	& \qquad +\sum_{k\neq i,j}(x_{2j-1}x_{2k-1}+x_{2j}x_{2k})U_{2k-1}+ \sum_{k\neq i,j}(x_{2j-1}x_{2k}-x_{2j}x_{2k-1})U_{2k} \\
	%& \qquad =-(1-x_{2i-1}^2-x_{2i}^2-x_{2j-1}^2-x_{2j}^2)U_{2j-1}+x_{2j-1}A-x_{2j}JA\\
	%  R^b(U_{2i-1},U_{2j-1})U_{2i} & = -(1-x_{2i-1}^2-x_{2i}^2-x_{2j-1}^2-x_{2j}^2)U_{2j} \\
	%  & \qquad -\sum_{k\neq i,j}(x_{2j-1}x_{2k}-x_{2j}x_{2k-1})U_{2k-1} + \sum_{k\neq i,j}(x_{2j-1}x_{2k-1}+x_{2j}x_{2k})U_{2k}\\
	%  %& \qquad = -(1-x_{2i-1}^2-x_{2i}^2-x_{2j-1}^2-x_{2j}^2)U_{2j} +x_{2j}A+x_{2j-1}JA \\
	R^b(U_{2i-1},U_{2j-1})U_{2j-1} & = (1-x_{2i-1}^2-x_{2i}^2-x_{2j-1}^2-x_{2j}^2)U_{2i-1} \\
	& \qquad -\sum_{k\neq i,j}(x_{2i-1}x_{2k-1}+x_{2i}x_{2k})U_{2k-1}+ \sum_{k\neq i,j}(x_{2i}x_{2k-1}-x_{2i-1}x_{2k})U_{2k} \\
	%  R^b(U_{2i-1},U_{2j-1})U_{2j} & = (1-x_{2i-1}^2-x_{2i}^2-x_{2j-1}^2-x_{2j}^2)U_{2i} \\
	%  & \qquad -\sum_{k\neq i,j}(x_{2i}x_{2k-1}-x_{2i-1}x_{2k})U_{2k-1}- \sum_{k\neq i,j}(x_{2i-1}x_{2k-1}+x_{2i}x_{2k})U_{2k} \\
	R^b(U_{2i-1},U_{2j-1})U_{2k-1} & = -(x_{2j-1}x_{2k-1}+x_{2j}x_{2k})U_{2i-1}+(x_{2j-1}x_{2k}-x_{2j}x_{2k-1})U_{2i}  \\
	& \qquad +(x_{2i-1}x_{2k-1}+x_{2i}x_{2k})U_{2j-1}+(x_{2i}x_{2k-1}-x_{2i-1}x_{2k})U_{2j} \\
	& \qquad + 2(x_{2i-1}x_{2j}-x_{2i}x_{2j-1})U_{2k} \quad (k\neq i,j)\\
	%  R^b(U_{2i-1},U_{2j-1})U_{2k} & = -(x_{2j-1}x_{2k}-x_{2j}x_{2k-1})U_{2i-1}-(x_{2j-1}x_{2k-1}+x_{2j}x_{2k})U_{2i} \\
	%  & \qquad -(x_{2i}x_{2k-1}-x_{2i-1}x_{2k})U_{2j-1}+(x_{2i-1}x_{2k-1}+x_{2i}x_{2k})U_{2j} \\
	%  & \qquad - 2(x_{2i-1}x_{2j}-x_{2i}x_{2j-1})U_{2k-1} \quad (k\neq i,j).
\end{align*}
and $R^b(U_{2i-1},U_{2j-1})U_{2r}=J(R^b(U_{2i-1},U_{2j-1})U_{2r-1})$ by Lemma~\ref{curvatura-JJ}(a). 
\end{proposition}

\

\begin{remark}
{\rm It follows from Lemma~\ref{curvatura-JJ}(c) that $R^b(U_{2i},U_{2j})=R^b(U_{2i-1},U_{2j-1})$ for any $i\neq j$.}
\end{remark}

%{\color{blue} Importante: hay relaciones entre las dos proposiciones anteriores, lo que puede hacer pensar en m\'as simetr\'ias de la curvatura:
%$$R^b(U_{2i-1},U_{2j})U_{2i-1} = R^b(U_{2i-1},U_{2j-1})U_{2i},\quad R^b(U_{2i-1},U_{2j})U_{2j-1} = -R^b(U_{2i-1},U_{2j-1})U_{2j}$$}
	
\medskip 
	
As a consequence of Propositions \ref{curvatura1}, \ref{curvatura3} and \ref{curvatura2}, we recover the familiar fact that the Bismut connection on $S^1\times S^3$ is flat:
	
\begin{corollary}\label{Rb-flat}
On $S^1\times S^3$ the Bismut connection is flat, i.e., $R^b\equiv 0$.
\end{corollary}
	
\begin{remark}
{\rm It was proved by Samelson in \cite{Sam} that any compact Lie group of even dimension admits a left invariant complex structure compatible with a bi-invariant metric. Moreover, it was shown in \cite{AI,J} that such Hermitian manifold is Bismut flat. We point out that $S^1\times S^{2n-1}$ (with $n\geq 2$) is a Lie group only for $n=2$. In this case we have that $S^1\times S^3$ is isomorphic to $S^1\times \SU(2)$, and the Hermitian structure $(J,g)$ is left-invariant (in fact, $g$ is bi-invariant), and it can be obtained with  Samelson's construction.}
\end{remark}
	
\
	
In what follows we will determine the holonomy group $\operatorname{Hol}^b(S^1\times S^{2n-1})$ of the Bismut connection $\nabla^b$. Since $S^1\times S^{2n-1}$ is connected, we can choose any point $p$ as base point, and it will be convenient for us to choose $p=\left(1,\frac{1}{\sqrt{2n}}(1,\ldots,1)\right)\in S^1\times S^{2n-1}$. We will use Theorem \ref{AS1} in order to determine its Lie algebra $\mathfrak{hol}^b$. We begin with some auxiliary results.

\begin{lemma}\label{Rb-1} 
For $n\geq 3$, we have that:
\begin{enumerate}
	\item[(a)] the set $\{R^b(U_{2i-1},U_{2j})_p \mid 1\leq i<j\leq n\}$ is linearly independent.
	\item[(b)] the set $\{R^b(U_{2i-1},U_{2j-1})_p \mid 1\leq i<j\leq n-1\}$ is linearly independent. 
\end{enumerate}
\end{lemma}
	
\begin{proof}
(a) It follows from Propositions \ref{curvatura3},  evaluating in $p=\left(1,\frac{1}{\sqrt{2n}}(1,\ldots,1)\right)$, that
%\begin{align*}
%	R^b(U_{2i-1},U_{2i})_p U_{2i-1} & =0,\\
%	R^b(U_{2i-1},U_{2i})_p U_{2r-1} & = \frac{2(n-2)}{n}U_{2r} -\frac{2}{n}\sum_{k\neq i,r}U_{2k}, 
%\end{align*}
%for $r\neq i$, while 
for $i\neq j$ we have 
\begin{align*}
	R^b(U_{2i-1},U_{2j})_p U_{2i-1} & = -\frac{n-2}{n}U_{2j}+\frac{1}{n}\sum_{k\neq i,j}U_{2k},\\
	R^b(U_{2i-1},U_{2j})_p U_{2j-1} & = -\frac{n-2}{n}U_{2i}+\frac{1}{n}\sum_{k\neq i,j}U_{2k},\\
	R^b(U_{2i-1},U_{2j})_p U_{2k-1} & = \frac{1}{n}U_{2i}+\frac{1}{n}U_{2j}-\frac{2}{n}U_{2k},
\end{align*}
for $k\neq i,j$. 
		
\smallskip
		
Let us consider a linear combination 
\[ \sum_{i<j} c_{ij} R^b(U_{2i-1},U_{2j})_p=0.\]
Expanding $g(\sum_{i<j} c_{ij} R^b(U_{2i-1},U_{2j})_p U_{2r-1},U_{2s})=0$ for $r<s$, we obtain 
\begin{equation}\label{system1}
	-(n-2)c_{rs}+\sum_{r<j\neq s} c_{rj}+\sum_{i<r}c_{ir}+\sum_{s<j}c_{sj}+\sum_{r\neq i<s} c_{is}=0, \qquad 1\leq r<s\leq n.
\end{equation}

Therefore \eqref{system1} defines a homogeneous linear system $Mc=0$  of $\binom{n}{2}$ equations with $\binom{n}{2}$ unknowns ordered lexicographically. It turns out that the matrix $M$ is symmetric, the elements on the diagonal are all equal to $-(n-2)$, the elements off the diagonal are equal to either $0$ or $1$, and all the rows and columns have the same sum, namely, $n-2$. The $\binom{n}{2}\times\binom{n}{2}$-matrix $\tilde M=(\tilde{M}_{ij})$ given by:
\[ \tilde{M}_{ij}=\begin{cases}
			-\frac{2n-6}{n(n-2)}, \quad \text{if } i=j,\\
			-\frac{n-6}{2n(n-2)}, \quad \text{if } M_{ij}=1,\\
			-\frac{2}{n(n-2)}, \quad \text{if } M_{ij}=0.
	\end{cases} \]
is the inverse of $M$, which means that the system $Mc=0$ has the unique solution $c_{ij}=0$ for all $i<j$. The proof of (a) is complete.

\medskip

(b) It follows from Proposition \ref{curvatura2},  evaluating in $p=\left(1,\frac{1}{\sqrt{2n}}(1,\ldots,1)\right)$, that when $i\neq j$ we have
\begin{align*}
	R^b(U_{2i-1},U_{2j-1})_p U_{2i-1} & = -\frac{n-2}{n}U_{2j-1}+\frac{1}{n}\sum_{k\neq i,j}U_{2k-1},\\
	R^b(U_{2i-1},U_{2j-1})_p U_{2j-1} & = \frac{n-2}{n}U_{2i-1}-\frac{1}{n}\sum_{k\neq i,j}U_{2k-1},\\
	R^b(U_{2i-1},U_{2j-1})_p U_{2r-1} & = -\frac{1}{n}U_{2i-1}+\frac{1}{n}U_{2j-1},
\end{align*}
for $r\neq i,j$. 

Let us consider now a linear combination 
\[ {\sum_{i<j}}' c_{ij} R^b(U_{2i-1},U_{2j-1})_p=0,\]
where ${\sum}'$ means that the indices run up to $n-1$, i.e. $1\leq i<j\leq n-1$. Expanding $\displaystyle{g({\sum_{i<j}}' c_{ij} R^b(U_{2i-1},U_{2j-1})_p U_{2r-1},U_{2s-1})=0}$ for $1\leq r<s\leq n-1$, we obtain
\begin{equation}\label{system2}
	-(n-2)c_{rs}+{\sum_{r<j\neq s}}'c_{rj}-{\sum_{i<r}}'c_{ir}-{\sum_{s<j}}'c_{sj}+{\sum_{r\neq i<s}}'c_{is}=0.
\end{equation}

Therefore \eqref{system2} defines a homogeneous linear system $Mc=0$  of $\binom{n-1}{2}$ equations with $\binom{n-1}{2}$ unknowns ordered lexicographically. It turns out that the matrix $M$ is symmetric, the elements on the diagonal are all equal to $-(n-2)$ and the elements off the diagonal are equal to either $0$ or $\pm 1$. The $\binom{n-1}{2}\times\binom{n-1}{2}$-matrix $\tilde M=(\tilde{M}_{ij})$ given by:
\[ \tilde{M}_{ij}=\begin{cases}
			-\frac{3}{n}, \quad \text{if } i=j,\\
			\mp\frac{1}{n}, \quad \text{if } M_{ij}=\pm1,\\
			0, \quad \text{if } M_{ij}=0.
		\end{cases} \]
is the inverse of $M$, and therefore $c_{ij}=0$ for all $1\leq i<j\leq n-1$. The proof of (b) is complete.
\end{proof}

\
	
Now we can prove the main result in this section.
	
\begin{theorem}
Let $n\geq 3$ and $\lambda>1$. If $S^1\times S^{2n-1}$ is equipped with the Vaisman structure $(J_\lambda,g_\lambda)$ described above, then the associated Bismut connection $\nabla^b$ has holonomy group $\operatorname{Hol}^b(S^1\times S^{2n-1})=\U(n-1)$.
\end{theorem}
	
\begin{proof}
We will compute, as before, the holonomy group $\operatorname{Hol}^b(S^1\times S^{2n-1})$ (and its Lie algebra $\mathfrak{hol}^b:=\mathfrak{hol}^b(S^1\times S^{2n-1})$) at the point $p=\left(1,\frac{1}{\sqrt{2n}}(1,\ldots,1)\right)\in S^1\times S^{2n-1}$.
		
Let us recall first from Corollary \ref{hol} that $\operatorname{Hol}^b(S^1\times S^{2n-1}) \subseteq \U(n-1)$. Therefore $\dim\mathfrak{hol}^b\leq (n-1)^2$.
		
Now, according to Theorem \ref{AS1},  $\mathfrak{hol}^b\subseteq \mathfrak{u}(n-1)$ contains all the curvature operators $R^b(X,Y)_p$ for any vector fields $X,Y$ on $S^1\times S^{2n-1}$. In particular, 
\begin{equation}\label{set}
\{R^b(U_{2i-1},U_{2j})_p\mid 1\leq i<j\leq n\} \cup \{R^b(U_{2i-1},U_{2j-1})_p\mid 1\leq i<j\leq n-1\} \subset \mathfrak{hol}^b.        \end{equation}
It follows from Lemma \ref{Rb-1} that each subset in the left-hand side of \eqref{set} is linearly independent; moreover, it is easy to see that their union is also linearly independent. Hence
\[ \dim\mathfrak{hol}^b\geq \binom{n}{2}+\binom{n-1}{2}= (n-1)^2.\]
Therefore $\dim\mathfrak{hol}^b=(n-1)^2$. This implies that $\mathfrak{hol}^b=\mathfrak{u}(n-1)$. Since $\operatorname{U}(n-1)$ is connected we have that $\operatorname{Hol}^b(S^1\times S^{2n-1})=\operatorname{U}(n-1)$.
\end{proof}

\medskip

We end this section this by writing down the Bismut Ricci curvature of the Hopf manifolds we are considering. The following result follows in a straightforward manner from Propositions \ref{curvatura1}, \ref{curvatura3} and \ref{curvatura2}.

\begin{proposition}
The Bismut Ricci curvature on $(S^1\times S^{2n-1},J_\lambda,g_\lambda)$ for $\lambda>1$ is given in terms of the orthonormal basis $\{U_1,\ldots,U_{2n}\}$ by
%\begin{align*}
%    \operatorname{Ric}^b(U_{2r-1},U_{2r-1}) & = -2(n-2)(x_{2r-1}^2+x_{2r}^2-1), \\
%    \operatorname{Ric}^b(U_{2r-1},U_{2s-1}) & = -2(n-2)(x_{2r-1}x_{2s-1}+x_{2r}x_{2s}), \\
%    \operatorname{Ric}^b(U_{2r-1},U_{2r}) & = 0, \\
%    \operatorname{Ric}^b(U_{2r-1},U_{2s}) & = -2(n-2)(x_{2r-1}x_{2s}-x_{2r}x_{2s-1}).
%\end{align*}
\begin{align*}
    \operatorname{Ric}^b(U_{2r-1},U_{2s-1}) & = -2(n-2)(x_{2r-1}x_{2s-1}+x_{2r}x_{2s} - \delta_{rs}), \\
    \operatorname{Ric}^b(U_{2r-1},U_{2s}) & = -2(n-2)(x_{2r-1}x_{2s}-x_{2r}x_{2s-1}).
\end{align*}
Also, $\Ric^b(U_{2r},U_{2s})=\Ric^b(U_{2r-1},U_{2s-1})$ for all $r,s$, according to Corollary \ref{ric-sym}.
\end{proposition}

\begin{corollary}
The Bismut connection associated to $(S^1\times S^{2n-1},J_\lambda,g_\lambda)$ has $\operatorname{Ric}^b=0$ if and only if $n=2$, and in this case $\nabla^b$ is flat (see Corollary \ref{Rb-flat}).
\end{corollary}

\
	
\section{Bismut holonomy of LCK Oeljeklaus-Toma manifolds}\label{section-OT}
	
In this section we study the Bismut holonomy of Oeljeklaus-Toma manifolds (OT manifolds for short) admitting an LCK metric. OT manifolds appeared in \cite{OT}, and they are non-K\"ahler compact complex manifolds which arise from certain number fields which admit $s$ real embeddings and $2t$ complex embeddings (OT manifolds of type $(s,t)$). When $t=1$ these OT manifolds admit LCK metrics. However, it was shown recently in \cite{DV, D, Vu} that they do not admit any LCK metric when $t\geq 2$.
	
It was proved in \cite{K} that all OT manifolds are in fact solvmanifolds, whose complex structure is induced by a left invariant one on the corresponding solvable Lie group. Using this solvmanifold structure, Kasuya also showed in \cite{K} that OT manifolds of any type do not admit Vaisman metrics. Moreover, for OT manifolds of type $(s,1)$, the LCK Hermitian structure is also induced by a left invariant one on the solvable Lie group. It can be deduced from \cite{K} that the Lie algebra $\frg$ of the Lie group corresponding to an OT manifold of type $(s,1)$ has a basis $\{A_1,\ldots,A_s,B_1,\ldots,B_s,C_1,C_2\}$ with Lie brackets given by\footnote{Note that for $s=1$ we obtain a Lie bracket isomorphic to the one in Example \ref{example}.}:
\begin{align*}
	[A_i,B_i] &=B_i, \qquad i=1,\ldots, s, \nonumber \\
	[A_i,C_1] &=-\frac12 C_1 +r_i C_2, \\
	[A_i,C_2] &=-r_i C_1 -\frac12 C_2, \nonumber
\end{align*}
for some real numbers $r_i\in\R$, $i=1,\ldots, s$. The complex structure $J$ on $\frg$ takes the following expression:
	\[ JA_i=B_i \; (i=1,\ldots, s), \qquad JC_1=C_2,\]
and the fundamental $2$-form is given by
	\[ \omega=2\sum_i \alpha_i\wedge\beta_i + \sum_{i\neq j}\alpha_i\wedge \beta_j + \gamma_1\wedge \gamma_2, \]
where $\{\alpha_1,\ldots,\alpha_s,\beta_1,\ldots,\beta_s,\gamma_1,\gamma_{2}\}$ is the dual basis of $\frg^*$. If $g=\omega(\cdot,J\cdot)$ denotes the corresponding Hermitian metric, then $(\frg,J,g)$ is an LCK Lie algebra, with Lee form $\theta=\alpha_1+\cdots+\alpha_s$. Note that the non-zero values of the metric $g$, in terms of the basis of $\frg$, are given by:
\[g(A_i,A_i)=g(B_i,B_i)=2,\quad g(A_i,A_j)=g(B_i,B_j)=1, \; (i\neq j),\quad g(C_1,C_1)=g(C_2,C_2)=1. \] 
Note that the vectors $A$ and $JA$, $g$-dual to the Lee form $\theta$ and $J\theta$ respectively, are given by
\[ A=\frac{1}{s+1}(A_1+\cdots+A_s), \quad JA=\frac{1}{s+1}(B_1+\cdots+B_s).\]
	
\

The main result of this section is Theorem \ref{hol-OT} where we set the holonomy group of OT manifolds of type $(s,1)$. Since these metrics are not Vaisman, one cannot expect a reduction of the holonomy group as stated in Corollary~\ref{hol}. In fact, in Theorem~\ref{hol-OT} we obtain the whole group $\operatorname{U}(s+1)$ as holonomy group. In order to get this result, several computations are needed.  We start determining the curvature endomorphisms $R^b$.
Using the well-known Koszul formula for the computation of the Levi-Civita connection $\nabla^g$, we obtain that $\nabla^g$, expressed in terms of the basis of $\frg$, is given by the following:
\begin{multicols}{2}
	\begin{itemize}
		\item $\nabla^g_{A_i}A_j=0, \; \forall i,j$
		 %\item $\nabla^g_{A_i}B_i=-\frac12 B_i+JA$
		\item $\nabla^g_{A_i}B_j=-\frac12 B_i+ \frac12(1+\delta_{ij}) JA$ 
		\item $\nabla^g_{A_i}C_1=r_i C_2$
		\item $\nabla^g_{A_i}C_2=-r_i C_1$
%		\item {\color{red}no necesario\footnote{Raquel: 9 diciembre 2020: usar que: $\nabla^g_XY - \nabla^g_YX = [X,Y]$}} $\nabla^g_{B_i}A_i=-\frac32B_i+JA$
		\item $\nabla^g_{B_i}A_j=-(\frac12+\delta_{ij}) B_j+\frac12(1+\delta_{ij}) JA$
		%\item $\nabla^g_{B_i}B_i=2A_i-2A$
		\item $\nabla^g_{B_i}B_j=(1+\delta_{ij})\left[\frac12(A_i+A_j)-A\right]$
		\item $\nabla^g_{B_i}C_k=0$, $k=1,2$
		\item $\nabla^g_{C_k}A_i=\frac12 C_k$, $k=1,2$
		\item $\nabla^g_{C_k}B_i=0$, $k=1,2$
		\item $\nabla^g_{C_k}C_k=-\frac12 A$, $k=1,2$
		\item $\nabla^g_{C_i}C_j=0$, $i\neq j$.
		%\item {\color{red}no necesario} $\nabla^g_{C_2}A_i=\frac12 C_2$
		%\item $\nabla^g_{C_2}B_i=0$
		%\item {\color{red}no necesario} $\nabla^g_{C_2}C_1=0$
		%\item $\nabla^g_{C_2}C_2=-\frac12 A$
	\end{itemize}
\end{multicols}
	%\begin{multicols}{2}
	%	\begin{itemize}
	%		\item $\nabla^g_{A_i}A_j=0 \; \forall i,j$
	%		\item $\nabla^g_{A_i}B_i=-\frac12 B_i+JA$
	%		\item $\nabla^g_{A_i}B_j=-\frac12 B_i+\frac12 JA, \; i\neq j$
	%		\item $\nabla^g_{A_i}C_1=r_i C_2$
	%		\item $\nabla^g_{A_i}C_2=-r_i C_1$
	%		\item $\nabla^g_{B_i}A_i=-\frac32B_i+JA$
	%		\item $\nabla^g_{B_i}A_j=-\frac12 B_j+\frac12 JA, \; i\neq j$
	%		\item $\nabla^g_{B_i}B_i=2A_i-2A$
	%		\item $\nabla^g_{B_i}B_j=\frac12(A_i+A_j)-A$
	%		\item $\nabla^g_{B_i}C_1=\nabla^g_{B_i}C_2=0$
	%		\item $\nabla^g_{C_1}A_i=\frac12 C_1$
	%		\item $\nabla^g_{C_1}B_i=0$
	%		\item $\nabla^g_{C_1}C_1=-\frac12 A$
	%		\item $\nabla^g_{C_1}C_2=0$
	%		\item $\nabla^g_{C_2}A_i=\frac12 C_2$
	%		\item $\nabla^g_{C_2}B_i=0$
	%		\item $\nabla^g_{C_2}C_1=0$
	%		\item $\nabla^g_{C_2}C_2=-\frac12 A$
	%	\end{itemize}
	%\end{multicols}
	
\
	
Applying \eqref{nablab-formula}, we arrive at the following result:
	
\begin{proposition}
The Bismut connection of an OT manifold of type $(s,1)$, expressed in terms of the basis $\{A_i,B_i,C_1,C_2\}$ of the corresponding Lie algebra $\frg$, is given by
\begin{multicols}{2}
		\begin{itemize}
		\item $\nabla^b_{A_i}A_j=0 \; \forall i,j$
					%\item $\nabla^g_{A_i}B_i=-\frac12 B_i+JA$
					%\item $\nabla^g_{A_i}B_j=-\frac12 B_i+\frac12 JA, \; i\neq j$
					\item $\nabla^b_{A_i}C_1=r_i C_2$
					%\item {\color{red}no necesario\footnote{Raquel, 9 diciembre 2020: en esta proposici\'on se puede usar que $\nabla^b_X JY = J(\nabla^b_X Y)$}} $\nabla^b_{A_i}C_2=-r_i C_1$
					%\item $\nabla^b_{B_i}A_i=-2B_i+2JA$
					\item $\nabla^b_{B_i}A_j=(1+\delta_{ij}) (-B_j + JA)$ 
					%\item {\color{red}no necesario} $\nabla^b_{B_i}B_i=2A_i-2A$
					%\item {\color{red}no necesario} $\nabla^b_{B_i}B_j=A_j-A$ 
					\item $\nabla^b_{B_i}C_1=-\frac12 C_2$
					%\item {\color{red}no necesario} $\nabla^b_{B_i}C_2=\frac12 C_1$
					\item $\nabla^b_{C_k}A_i=\frac12 C_k$, $k=1,2$
					%\item {\color{red}no necesario} $\nabla^b_{C_1}B_i=\frac12 C_2$
					\item $\nabla^b_{C_1}C_1=-\frac12 A$
					%\item {\color{red}no necesario} $\nabla^b_{C_1}C_2=-\frac12 JA$
					%\item $\nabla^b_{C_2}A_i=\frac12 C_2$
					%\item {\color{red}no necesario} $\nabla^b_{C_2}B_i=-\frac12 C_1$
					\item $\nabla^b_{C_2}C_1=\frac12 JA$
					%\item {\color{red}no necesario} $\nabla^b_{C_2}C_2=-\frac12 A$
				\end{itemize}
			\end{multicols}
The missing values can be deduced from 	$\nabla^b_X JY = J(\nabla^b_X Y)$.
	%	\begin{multicols}{2}
	%		\begin{itemize}
	%			\item $\nabla^b_{A_i}A_j=\nabla^b_{A_i}B_j=0 \; \forall i,j$
	%			%\item $\nabla^g_{A_i}B_i=-\frac12 B_i+JA$
	%			%\item $\nabla^g_{A_i}B_j=-\frac12 B_i+\frac12 JA, \; i\neq j$
	%			\item $\nabla^b_{A_i}C_1=r_i C_2$
	%			\item $\nabla^b_{A_i}C_2=-r_i C_1$
	%			\item $\nabla^b_{B_i}A_i=-2B_i+2JA$
	%			\item $\nabla^b_{B_i}A_j=- B_j+JA, \; i\neq j$
	%			\item $\nabla^b_{B_i}B_i=2A_i-2A$
	%			\item $\nabla^b_{B_i}B_j=A_j-A$
	%			\item $\nabla^b_{B_i}C_1=-\frac12 C_2$
	%			\item $\nabla^b_{B_i}C_2=\frac12 C_1$
	%			\item $\nabla^b_{C_1}A_i=\frac12 C_1$
	%			\item $\nabla^b_{C_1}B_i=\frac12 C_2$
	%			\item $\nabla^b_{C_1}C_1=-\frac12 A$
	%			\item $\nabla^b_{C_1}C_2=-\frac12 JA$
	%			\item $\nabla^b_{C_2}A_i=\frac12 C_2$
	%			\item $\nabla^b_{C_2}B_i=-\frac12 C_1$
	%			\item $\nabla^b_{C_2}C_1=\frac12 JA$
	%			\item $\nabla^b_{C_2}C_2=-\frac12 A$
	%		\end{itemize}
	%	\end{multicols}
	\end{proposition}
	
\

Using the previous proposition we obtain the following expressions for the curvature $R^b$.  The proof consists of long but standard computations.
	
\begin{proposition}\label{Rb-OT}
The curvature $R^b$ of the Bismut connection on an OT manifold of type $(s,1)$ is given by:
\begin{multicols}{2}
	\begin{itemize}
	\item $R^b(A_i,A_j)=0$
	\item $R^b(A_i,B_j)=0  \; (i\neq j)$
	%\item $R^b(A_i,B_i)A_i=2(B_i-JA)$
	\item $R^b(A_i,B_i)A_j=(1+\delta_{ij})(B_j-JA)$
	\item $R^b(A_i,B_i)C_1=\frac12 C_2$
	%\item $R^b(B_i,B_j)A_i=\frac{1}{s+1}(A_i-2A_j)$
	%\item {\color{red}no necesario\footnote{Se deduce de la anterior, cambiando $i$ por $j$}}$R^b(B_i,B_j)A_j=-\frac{1}{s+1}(A_j-2A_i)$
	\item $R^b(B_i,B_j)A_k=\frac{(1+\delta_{jk})A_i-(1+\delta_{ik})A_j}{s+1}$%\left((1+\delta_{jk})A_i-(1+\delta_{ik})A_j\right) $
	\item $R^b(B_i,B_j)C_1=0$
	\item $R^b(A_i,C_k)A_j=\frac14 C_k \; (\forall i,j,k)$
	\item $R^b(A_i,C_1)C_1=-\frac14 A$
	%\item $R^b(A_i,C_2)A_j=\frac14 C_2 \; (\forall i,j)$
	\item $R^b(A_i,C_2)C_1=\frac14 JA$
	%\item $R^b(B_i,C_1)A_i=\frac{-s+3}{4(s+1)}C_2$
	\item $R^b(B_i,C_1)A_j=\frac{-s+1+2\delta_{ij}}{4(s+1)}C_2$
	\item $R^b(B_i,C_1)C_1=\frac{1}{2(s+1)}B_i-\frac14 JA$
	%\item $R^b(B_i,C_2)A_i=\frac{s-3}{4(s+1)}C_1$
	\item $R^b(B_i,C_2)A_j=\frac{s-1-2\delta_{ij}}{4(s+1)}C_1$
	\item  $R^b(B_i,C_2)C_1=\frac{1}{2(s+1)}A_i-\frac14 A$
	\item $R^b(C_1,C_2)A_i=-\frac12 JA$
	\item $R^b(C_1,C_2)C_1=\frac{s}{2(s+1)} C_2$
	\end{itemize}
\end{multicols}
	%	\begin{multicols}{2}
	%		\begin{itemize}
	%			\item $R^b(A_i,A_j)=0$
	%			\item $R^b(A_i,B_j)=0  \; (i\neq j)$
	%			\item $R^b(A_i,B_i)A_i=2(B_i-JA)$
	%			\item $R^b(A_i,B_i)A_j=B_j-JA \; (i\neq j)$
	%			\item $R^b(A_i,B_i)C_1=\frac12 C_2$
	%			\item $R^b(B_i,B_j)A_i=\frac{1}{s+1}(A_i-2A_j)$
	%			\item $R^b(B_i,B_j)A_j=-\frac{1}{s+1}(A_j-2A_i)$
	%			\item $R^b(B_i,B_j)A_k=\frac{1}{s+1}(A_i-A_j) \; (k\neq i,j)$
	%			\item $R^b(B_i,B_j)C_1=0$
	%			\item $R^b(A_i,C_1)A_j=\frac14 C_1 \; (\forall i,j)$
	%			\item $R^b(A_i,C_1)C_1=-\frac14 A$
	%			\item $R^b(A_i,C_2)A_j=\frac14 C_2 \; (\forall i,j)$
	%			\item $R^b(A_i,C_2)C_1=\frac14 JA$
	%			\item $R^b(B_i,C_1)A_i=\frac{-s+3}{4(s+1)}C_2$
	%			\item $R^b(B_i,C_1)A_j=\frac{-s+1}{4(s+1)}C_2\; (i\neq j)$
	%			\item $R^b(B_i,C_1)C_1=\frac{1}{2(s+1)}B_i-\frac14 JA$
	%			\item $R^b(B_i,C_2)A_i=\frac{s-3}{4(s+1)}C_1$
	%			\item $R^b(B_i,C_2)A_j=\frac{s-1}{4(s+1)}C_1\; (i\neq j)$
	%			\item $R^b(B_i,C_2)C_1=\frac{1}{2(s+1)}A_i-\frac14 A$
	%			\item $R^b(C_1,C_2)A_i=-\frac12 JA$
	%			\item $R^b(C_1,C_2)C_1=\frac{s}{2(s+1)} C_2$
	%		\end{itemize}
	%	\end{multicols}
The missing values are either $0$ or can be deduced from the ones in this list using that $R^b(\cdot,\cdot)J=JR^b(\cdot,\cdot)$.
\end{proposition}
	
\medskip

\begin{remark}
{\rm It follows from Proposition \ref{Rb-OT} that the curvature operators $R^b(A_i,C_1)$ and $R^b(A_i,C_2)$ are independent of $i$. We will denote simply
\[ R^b_{AC1}:=R^b(A_i,C_1), \qquad   R^b_{AC2}:=R^b(A_i,C_2),\]
for any $i$. Moreover, for $s\neq 2$, these operators verify a linear relation with other curvature operators, since
\[
R^b_{AC1}=\frac{1}{s-2}\sum_i R^b(B_i,C_2), \qquad 
R^b_{AC2}=-\frac{1}{s-2}\sum_i R^b(B_i,C_1).
\] }
\end{remark}	

\
	
We may now compute the holonomy algebra $\mathfrak{hol}^b(M)$ of the Bismut connection associated to the LCK structure on an OT manifold $M=\Gamma\backslash G$ of type $(s,1)$. Recall that $\mathfrak{hol}^b(M)$ is the smallest subalgebra of $\mathfrak{u}(\frg)\cong \mathfrak{u}(s+1)$ containing the curvature operators $R^b(X,Y)$, $X,Y\in\frg$, and being  closed under commutators with $\nabla^b_X$, $X\in\frg$, due to Theorem \ref{AS2}.

\

For low dimensions, it is straightforward to verify that $\mathfrak{hol}^b(M)= \mathfrak{u}(\frg)$, that is, there is no reduction of the Bismut holonomy group. Indeed, computing all the corresponding curvature operators and the commutators between any two of them we obtain:

\begin{lemma}\label{dim-baja}
For $s=1$ and $s=2$, the holonomy algebra $\mathfrak{hol}^b(M)$ of the Bismut connection on an OT manifold $M$ of type $(s,1)$ coincides with $\mathfrak{u}(\frg)$.
\end{lemma}

Note that the computations for the case $s=1$ were carried out in Example \ref{example}. 

\

Therefore, we assume from now on that $s\geq 3$. Our strategy for computing the holonomy consist on finding two sets of linearly independent homomorphisms belonging to $\mathfrak{hol}^b(M)$ (see Lemma~\ref{li-1} and Lemma~\ref{li-2}) with a suitable number of elements.

%Our first result in this case is the following:

\begin{lemma}\label{li-1}
The elements of the subset \[ \mathcal{U}:=\{R^b(B_i,C_1), R^b(B_i,C_2)\}_{1\leq i\leq s} \cup \{R^b(B_i,B_j)\}_{1\leq i<j\leq s}\] 
of $\mathfrak{hol}^b(M)$ are linearly independent. \end{lemma}

\begin{proof}
Consider a linear combination of elements of $\mathcal U$
\[ T:= \sum_{1\leq i<j\leq s} x_{ij}R^b(B_i,B_j)+\sum_{i=1}^s c_i R^b(B_i,C_1)+\sum_{i=1}^s d_i R^b(B_i,C_2),  \]
and assume $T=0$. 

For any $k=1,\ldots,s$, we look for the coefficient of $C_1$ in $T(A_k)=0$ and we obtain, according to Proposition \ref{Rb-OT},
\[  (s-3)d_k+(s-1)\sum_{i\neq k} d_i=0.  \]
This implies
\begin{equation}\label{dk} 
	-2d_k+(s-1)\sum_i d_i=0,
\end{equation}
and summing this equality over all $k=1,\ldots,s$, we get
\[ -2\sum_k d_k+s(s-1)\sum_k d_k=0, \]
which is equivalent to
\[ (s+1)(s-2)\sum_k d_k=0.\]
Since $s\geq 3$, we deduce that $\sum_k d_k=0$, which together with \eqref{dk} gives $d_k=0$ for all $k$. 

Taking into account the coefficient of $C_2$ in $T(A_k)=0$  we obtain, in a similar fashion, that $c_k=0$ for all $k$.

Now, $T(A_k)=0$ is equivalent to $\sum_{i<j} x_{ij} R^b(B_i,B_j)(A_k)=0$, which when expanded gives
\[ \sum_{i<j} x_{ij}(A_i-A_j)+\sum_{i<k} x_{ik} A_i-\sum_{j>k} x_{kj} A_j=0,  \]
for all $k$. Note that the first sum in the equation above is independent of $k$, so that 
\[ \sum_{i<k} x_{ik} A_i-\sum_{j>k} x_{kj} A_j=v \]
for some constant vector $v\in\text{span}\{A_1,\ldots,A_s\}$, for any $k$. Fix now a pair $(p,q)$ with $1\leq p<q\leq s$, we have then
\[  \sum_{i<p} x_{ip} A_i-\sum_{j>p} x_{pj} A_j= \sum_{i<q} x_{iq} A_i-\sum_{j>q} x_{qj} A_j.\]
The coefficient of $A_q$ in the left-hand side is $-x_{pq}$, whereas on the right-hand side is $0$. Therefore $x_{pq}=0$ for all $1\leq p<q\leq s$, and the proof is complete.
\end{proof}

\

According to Lemma \ref{li-1}, we can forget about the operators $R^b_{AC1}$ and $R^b_{AC2}$ when searching for a basis of $\mathfrak{hol}^b(M)$.
We should also analyze the commutators between any two curvature operators.  We will prove that we need not compute all these commutators, but only the ones appearing in the next result.

\ 

For any $i=1,\ldots,s$, let $S_i$ denote the endomorphism of $\frg$ defined as follows:
\[ S_i(A_j)=\delta_{ij}\left(-sB_j+\sum_{k\neq j} B_k\right), \qquad S_i(C_1)=-\frac12 C_2, \qquad S_iJ=JS_i.\]  
It is easy to verify that $S_i$ is skew-symmetric, and therefore $S_i\in \mathfrak{u}(\frg)$. Moreover, the following result relates them with the curvature operators $R^b(A_i,B_i)$:

\begin{lemma}\label{Si}
The endomorphims $\{S_i \}_{1\leq i\leq s}$ in $\mathfrak{u}(\frg)$ are linearly independent and, furthermore, 
\[ \text{span}\{S_i \}_{1\leq i\leq s}=\text{span}\{R^b(A_i,B_i)\}_{1\leq i\leq s}.\]
In particular, $S_i\in\mathfrak{hol}^b(M)$.
\end{lemma}

\begin{proof}
The fact that $\{S_1,\ldots, S_s\}$ are linearly independent is very easy to verify. As for the second statement, it is a consequence of the following expressions:
\[ R^b(A_i,B_i)=-\frac{1}{s+1}\left(2S_i+\sum_{k\neq i} S_k\right),\]
and 
\[ S_i=-sR^b(A_i,B_i)+\sum_{k\neq i}R^b(A_k,B_k).\]
\end{proof}

\ 

Let us consider now the endomorphisms $T_{ij}\in\mathfrak{hol}^b(M)$, $i\neq j$, defined by
\[ T_{ij}=[S_i,R^b(B_i,B_j)]. \]
Direct computations prove the following:

\begin{lemma}\label{Tij-new}
The operators $T_{ij}$, $i\neq j$, act on $\frg$ as follows:
\begin{itemize}
	\item $T_{ij}(A_i)=\frac{1}{s+1}\left(-sB_i-s B_j+\sum_{k\neq i,j}B_k\right)$, 
	\item $T_{ij}(A_j)=\frac{1}{s+1}\left(-2s B_i+2\sum_{k\neq i}B_k\right)$,
	\item $T_{ij}(A_l)=\frac{1}{s+1}\left(-sB_i+\sum_{k\neq i}B_k\right)$, for $l\neq i,j$,
	\item $T_{ij}(C_1)=0$.
\end{itemize}
The missing values can be deduced from $T_{ij}J=JT_{ij}$.
\end{lemma}		

\medskip

\begin{lemma}\label{li-2}
The elements of the subset \[ \mathcal{V}:=\{S_i\}_{1\leq i\leq s}\cup \{R^b(C_1,C_2)\} \cup\{T_{ij}\}_{1\leq i<j\leq s}\] of $\mathfrak{hol}^b(M)$ 
are linearly independent.
\end{lemma}

\begin{proof}
Let us consider a linear combination of elements of $\mathcal V$
\[ S:=\sum_{i=1}^s x_i S_i + y R^b(C_1,C_2) + \sum_{1\leq i<j\leq s} a_{ij}T_{ij},\]
for some $x_i,y,a_{ij}\in\R$, and assume $S=0$. %Note that $S$ commutes with $J$.
Let us see first that $y=0$.

It follows from $S(C_1)=0$ that 
\begin{equation}\label{suma-x}
	\sum_i x_i=\frac{s}{s+1}y. 
\end{equation}

Fix now $l\in\{1,\ldots,s\}$, and consider $S(A_l)=0$: 
\begin{equation}\label{SAl}
\begin{array}{lll}
S(A_l)& =& x_l(-s\,B_l + \displaystyle\sum_{k\neq l} B_k) - \frac{y}{2(s+1)}(B_1+\cdots+B_s) \\[5pt] && + \frac{1}{s+1} \sum_{i<l} a_{il} (-2s B_i + 2 \displaystyle\sum_{k\neq i} B_k)\\[5pt]
&& + \frac{1}{s+1} \displaystyle\sum_{j>l} a_{lj} (-s B_l -s B_j + \displaystyle\sum_{k\neq l,j} B_k) + \frac{1}{s+1} \displaystyle\sum_{l\neq i<j \neq l} a_{ij} (-s B_i +  \displaystyle\sum_{k\neq i} B_k)\\& =& 0.
\end{array}
\end{equation}

The coefficient of $B_l$ in \eqref{SAl} is zero and, using Lemmas \ref{Si} and \ref{Tij-new}, it is given by
\begin{equation}\label{Bl} 
-s x_l-\frac{y}{2(s+1)}+\frac{2}{s+1}X_l-\frac{s}{s+1}Y_l+\frac{1}{s+1}Z_l=0, 
\end{equation}
where 
\[ X_l=\sum_{i<l} a_{il}, \qquad Y_l=  \sum_{j>l} a_{lj}, \qquad Z_l=\sum_{l\neq i<j\neq l}a_{ij}. \]
Note that $X_l+Y_l+Z_l=\sum_{i<j}a_{ij}$ for all $l$ and, moreover,
\[ \sum_{l=1}^s X_l=\sum_{i<j}a_{ij}, \quad \sum_{l=1}^s Y_l=\sum_{i<j}a_{ij}, \quad \sum_{l=1}^s Z_l=(s-2)\sum_{i<j}a_{ij}.\]
Summing \eqref{Bl} over $l=1,\ldots,s$ and using \eqref{suma-x} we arrive at
\[ -s\frac{s}{s+1}y-s\frac{y}{2(s+1)}+\frac{2}{s+1}
\sum_{i<j}a_{ij}-\frac{s}{s+1}\sum_{i<j} a_{ij}+\frac{s-2}{s+1}\sum_{i<j}a_{ij}=0. \]
From this we deduce
\[  \frac{-s(2s+1)}{2(s+1)}y=0, \]
thus 
\[ y=0. \]

Fixing again $l\in\{1,\ldots,s\}$, we compute now the coefficient of $B_r$ in \eqref{SAl}. Using Proposition \ref{Rb-OT}, Lemmas \ref{Si} and \ref{Tij-new} and $y=0$ we obtain the system:

\begin{align}
	-s x_l+\frac{2}{s+1}X_l-\frac{s}{s+1}Y_l+\frac{1}{s+1}Z_l & =0,  \qquad \text{for } r=l, \nonumber\\
	x_l-a_{rl}+\frac{2}{s+1}X_l+\frac{1}{s+1}Y_l+\frac{1}{s+1}Z_l & =Y_r, \qquad \text{for } r<l,\label {system}\\
	x_l-a_{lr}+\frac{2}{s+1}X_l+\frac{1}{s+1}Y_l+\frac{1}{s+1}Z_l & =Y_r, \qquad \text{for } r>l.\nonumber
\end{align}

Summing all the equations in the system \eqref{system} we obtain
\begin{equation*}
-x_l-X_l-Y_l+\frac{2s}{s+1}X_l-\frac{1}{s+1}Y_l+\frac{s}{s+1}Z_l= \sum_{r\neq l} Y_r,
\end{equation*}
and, since $\sum_{r\neq l}Y_r=\sum_{r}Y_r-Y_l=\sum_{i<j}a_{ij}-Y_l$, we deduce
\begin{equation}\label{xl}
	x_l=\frac{s-1}{s+1}X_l-\frac{1}{s+1}Y_l+\frac{s}{s+1}Z_l-\sum_{i<j}a_{ij}.
\end{equation}
Summing \eqref{xl} over $l=1,\ldots, s$, and recalling \eqref{suma-x} with $y=0$, we get
\[ 0=\sum_l x_l =\sum_{l=1}^s\left(\frac{s-1}{s+1}X_l-\frac{1}{s+1}Y_l+\frac{s}{s+1}Z_l -\sum_{i<j}a_{ij} \right)=-2\sum_{i<j} a_{ij}. \]
Thus,
\begin{equation}\label{suma-a}
	\sum_{i<j} a_{ij}=0,
\end{equation}
which is equivalent to 
%\begin{equation}\label{eq-1}
\[	X_l+Y_l+Z_l=0,  \]
%\end{equation}
for all $l$. Now, replacing $Z_l=-X_l-Y_l$ in \eqref{xl} and using \eqref{suma-a} we arrive at
\begin{equation}\label{xl-bis}
	x_l=-\frac{1}{s+1}X_l-Y_l.
\end{equation}
Now, if we replace this value of $x_l$ in the first equation of the system \eqref{system}, we arrive at
\begin{equation}\label{eq-2}
	X_l+(s-1)Y_l=0,
\end{equation}
for all $l=1,\ldots,s$.

\medskip

\textsl{Claim:} $X_l=Y_l=Z_l=0$ for all $l=1,\ldots,s$. Observe that it suffices to prove that $X_l = Y_l = 0$.

\smallskip

The proof of the claim follows by induction on $l$. We begin with the case $l=1$. It is clear that $X_1=0$, and it follows from \eqref{eq-2} for $l=1$ that $Y_1=0$. The case $l=1$ is proved.

Next, fix $1<l\leq s$ and assume that $X_r=Y_r=Z_r=0$ for all $r<l$. With this hypothesis, the equations corresponding to $r<l$ in the system \eqref{system} can be written as
%\begin{equation}\label{eq-arl}
\[ 	  x_l-a_{rl}+\frac{1}{s+1}X_l=0. \]
%\end{equation}
Substituting the value of $x_l$ from \eqref{xl-bis}, we obtain that 
\[ Y_l=- a_{rl}. \]
Summing over $r<l$, we obtain
\[ (l-1)Y_l = -X_l,\] 
which together with \eqref{eq-2} gives 
\[ X_l=Y_l=0, \quad \text{for} \quad l\leq s-1.\]
For $l=s$ we simply need to point out that $Y_s=0$, due to the very definition of $Y_s$. It follows from \eqref{eq-2} that $X_s=0$ and hence the claim is proved.

\

We notice next that the claim just proved and \eqref{xl-bis} imply that $x_l=0$ for all $l$. Now it is clear from the system \eqref{system} that $a_{ij}=0$ for all $i<j$.
\end{proof}

\

\begin{lemma}\label{lema-final}
The subset $\mathcal{U}\cup \mathcal{V}$ of $\mathfrak{hol}^b(M)$ is linearly independent.
%\[  \{R^b(A_i,B_i)\}_{i}\cup\{T_{ij}\}_{i<j} \cup  \{R^b(C_1,C_2)\} \cup \{R^b(B_i,C_1), R^b(B_i,C_2)\}_{i} \cup \{R^b(B_i,B_j)\}_{i<j} . \]	
\end{lemma}

\begin{proof}
Analyzing the action on $\frg$ of each of the operators in $\mathcal{U}\cup \mathcal{V}$, it is easy to verify that the linear independence of $\mathcal{U}\cup \mathcal{V}$ is equivalent to the linear independence of $\mathcal{U}$ and $\mathcal{V}$ separately. Thus this result follows from Lemmas \ref{li-1} and \ref{li-2}.
\end{proof}

\

With these lemmas we are able to finally prove the main result of this section.

\begin{theorem}\label{hol-OT}
The holonomy group of the Bismut connection $\nabla^b$ on an OT manifold $M$ of type $(s,1)$ (hence of dimension $2s+2$) is  $\operatorname{Hol}^b(M)=\operatorname{U}(s+1)$, for any $s\in\N$. Therefore, there is no reduction of the Bismut  holonomy. \end{theorem}

\begin{proof}
If $s=1,2$, then $\mathfrak{hol}^b(M)= \mathfrak{u}(s+1)$, according to Lemma \ref{dim-baja}.

For $s\geq 3$, the cardinal of the subset $\mathcal{U}$ is $2s + \binom{s}{2} = \frac{s^2+3s}{2}$, whereas the cardinal of  $\mathcal{V}$ is $s + 1 + \binom{s}{2} = \frac{s^2+s+2}{2}$. Therefore, the cardinal of the subset of $\mathfrak{hol}^b(M)$ given in Lemma \ref{lema-final} is \[\frac{s^2+3s}{2} + \frac{s^2+s+2}{2} = (s+1)^2,\]
%\[ s+\binom{s}{2}+ 1+2s+\binom{s}{2}=s(s-1)+3s+1=s^2+2s+1=(s+1)^2, \] 
so that $\dim\mathfrak{hol}^b(M)\geq (s+1)^2$. On the other hand, we know that $\mathfrak{hol}^b(M)$ is a subalgebra of $\mathfrak{u}(\frg)\cong \mathfrak{u}(s+1)$, so that $\dim\mathfrak{hol}^b(M)\leq (s+1)^2$. Therefore we arrive at 
\[  \mathfrak{hol}^b(M)=  \mathfrak{u}(\frg)\cong \mathfrak{u}(s+1).\]

Since $\operatorname{U}(s+1)$ is connected, this implies that $\operatorname{Hol}^b(M)=\operatorname{U}(s+1)$, for all $s\geq 1$.
\end{proof}

\ 

\section{Gauduchon connections and the Vaisman condition}

In the last section of the article we study the relation between the Gauduchon connections on an LCK manifold and the Vaisman condition. In fact, we prove that if the Lee form of a compact LCK manifold is non-zero and parallel with respect to a Gauduchon connection $\nabla^t$, then the LCK manifold is Vaisman and, moreover, $t=-1$. In other words, the Lee form can only be parallel with respect to the Bismut connection and in this case it is also parallel with respect to the Levi-Civita connection (recall Theorem \ref{theta-parallel}).

\begin{theorem}\label{gauduchon}
Let $(M,J,g)$ be a connected compact LCK manifold with corresponding Lee form $\theta$ and let $\{\nabla^t\}_{t\in\R}$ be the family of associated Gauduchon connections \eqref{canonical}. If $\nabla^t\theta=0$ for some $t\in\R$ then either:
\begin{enumerate}
    \item[(a)] $\theta=0$, i.e., $(M,J,g)$ is K\"ahler (and therefore $\nabla^t=\nabla^g$ for all $t$), or
    \item[(b)] $(M,J,g)$ is Vaisman and $t=-1$ (therefore $\nabla^{-1}=\nabla^{b}$ is the Bismut connection).
\end{enumerate}
\end{theorem}

\begin{proof}
As usual, let us denote by $A$ the vector field on $M$ which is $g$-dual to $\theta$. Then it follows from \eqref{canonical} that, for any $X,Y\in\mathfrak{X}(M)$,
\begin{align*}
    (\nabla^t_X\theta)(Y) & = g(\nabla^t_X A,Y) \\
                          & = g(\nabla^g_X A,Y)-\frac{t-1}{4} d\omega (JX,JA,JY) -\frac{t+1}{4}d\omega (JX,A,Y).
\end{align*}
From \eqref{3-form} we obtain that $d\omega(JX,JA,JY)=c(X,A,Y)$, where $c$ is the torsion $3$-form of the corresponding Bismut connection. Taking now into account Corollary \ref{torsion-A} we arrive at $d\omega(JX,JA,JY)=0$ for any $X,Y$. As a consequence the expression for $(\nabla^t_X\theta)(Y)$ becomes
\begin{align*}
    (\nabla^t_X\theta)(Y) & = g(\nabla^g_X A,Y)- \frac{t+1}{4}\theta\wedge \omega (JX,A,Y)\\
                          & =g(\nabla^g_X A,Y)- \frac{t+1}{4}\left(\theta(JX)g(JA,Y)+|A|^2g(X,Y)-g(A,X)g(A,Y)\right).
\end{align*}
Therefore, if $\nabla^t\theta=0$ for some $t\in\R$ then 
\begin{equation}\label{nabla-t}
    \nabla^g_X A = \frac{t+1}{4}\left(|A|^2X-\theta(X)A+\theta(JX)JA\right)
\end{equation}
for any vector field $X$ on $M$. Note that the $(1,1)$-tensor $\nabla^g A$ is symmetric (in accordance with $\theta$ being closed) and, moreover, it commutes with the complex structure $J$. Hence, it follows from \cite[Lemma 3]{MMO} that $A$ is holomorphic, i.e., $\mathcal{L}_AJ=0$.

\smallskip 

Next, we observe that, since $\nabla^t$ is a metric connection and $\theta$ is $\nabla^t$-parallel, $|A|$ is a constant function. That is, $|A|=c$ for some $c\in \R$, $c\geq 0$. If $c=0$ then $\theta=0$ and therefore $(M,J,g)$ is K\"ahler.
%We will show next that $|A|$ is a constant function. For any $p\in M$ and $v\in T_pM$, let us compute $v(|A|^2)$. We have two possible cases.

%\smallskip

%(i) $A_p=0$. Clearly, \eqref{nabla-t} implies $\nabla^g_v A=0$ for any $v\in T_pM$. Therefore,
%\[ v(|A|^2)=v\left(g(A,A)\right)=2g_p(\nabla^g_vA,A_p)=0. \]

%\smallskip 

%(ii) $A_p\neq 0$. Since $\theta(JA)=0$ and $\theta(A)=|A|^2$, we obtain from \eqref{nabla-t} that $\nabla^g_A A =\nabla^g_{JA} A =0$ and hence
%\[ v(|A|^2)=0 \qquad \text{for} \qquad v=A_p, \, v=(JA)_p.\]
%Next, consider $v\in T_pM$ such that $g_p(v,A_p)=g_p(v,(JA)_p)=0$, or equivalently, $\theta_p(v)=\theta_p(J_pv)=0$. As a consequence,
%\[ v(|A|^2)=2g_p(\nabla^g_vA,A_p)= 2g_p\left(\frac{t+1}{4} |A_p|^2 v,A_p\right)=0, \]
%using \eqref{nabla-t}.

%\smallskip 

%As $M$ is connected, this implies that 

\smallskip

Assume now $c>0$. We have proved that $A$ is holomorphic with constant length. According to \cite[Theorem 1(i)]{MMO}, we have that the compact LCK manifold $(M,J,g)$ is actually Vaisman, that is, $\nabla^gA=0$. Choosing $0\neq v\in T_pM$ for some $p\in M$ with $\theta_p(v)=\theta_p(J_pv)=0$, it follows from \eqref{nabla-t} that 
\[ 0=\nabla^g_vA=\frac{t+1}{4}c^2v.\]
Thus $t=-1$, and the proof is complete.
\end{proof}

\medskip

\begin{example}
{\rm There exist compact LCK manifolds which admit a Hermitian connection with respect to which the Lee form is parallel, but the Hermitian structure is not Vaisman. Indeed, consider a solvmanifold $M=\Gamma\backslash G$ from Example \ref{example} equipped with the LCK structure exhibited there. We can define a Hermitian connection $\nabla$ on $G$ in terms of the basis $\{e_1,\ldots,e_4\}$ of left invariant vector fields simply by setting
\[ \nabla_{e_1}e_3=e_4, \quad \nabla_{e_1}e_4=-e_3, \quad \nabla_{e_i}e_j=0,\]
for all other possible choices of $(i,j)$. The Lee form on $G$ is parallel with respect to $\nabla$ and the same happens on $M$ with the induced connection. 
}
\end{example}

\ 

We can generalize Theorem \ref{gauduchon} to a larger class of metric connections. Indeed, in \cite{OUV} a $2$-parameter family $\{\nabla^{\varepsilon,\rho}\mid (\varepsilon,\rho)\in \R^2\}$ of metric connections was introduced on any Hermitian manifold. Inspired by formula \eqref{canonical}, these connections are defined by 
\begin{equation}\label{e-r} 
	g(\nabla^{\varepsilon,\rho}_XY,Z)=g(\nabla^g_XY,Z)-\varepsilon\, d\omega(JX,JY,JZ)-\rho\, d\omega(JX,Y,Z). 
\end{equation}
Note that the Gauduchon connections $\nabla^t$ correspond to $\nabla^{\varepsilon,\rho}$ with $\varepsilon+\rho=\frac12$ and $t=1-4\varepsilon$. In particular, $\nabla^b=\nabla^{1/2,0}$; moreover, $\nabla^g=\nabla^{0,0}$.

It is clear that all these connections are metric, i.e., $\nabla^{\varepsilon,\rho}g=0$, since the expression $\varepsilon\, d\omega(JX,JY,JZ)+\rho\, d\omega(JX,Y,Z)$ is skew-symmetric in $Y,Z$. However, it is not true that they are all compatible with $J$: it was proved in \cite{OUV} that 
\[ \nabla^{\varepsilon,\rho}J= -2\left(\varepsilon+\rho-\frac12\right)\nabla^gJ;\]
therefore, if $(M,J,g)$ is not K\"ahler, then $\nabla^{\varepsilon,\rho}$ is a Hermitian connection if and only if $\varepsilon+\rho=\frac12$, i.e., it is a Gauduchon connection.

\medskip 

With the exact same proof of Theorem \ref{gauduchon} we can show the following result.

\begin{theorem}
Let $(M,J,g)$ be a connected compact LCK manifold with corresponding Lee form $\theta$ and consider the metric connections $\nabla^{\varepsilon,\rho}$ on $M$ defined as in \eqref{e-r}. If $\nabla^{\varepsilon,\rho}\theta=0$ for some $(\varepsilon,\rho)\in \R^2$ then either:
\begin{enumerate}
    \item[(a)] $\theta=0$, i.e., $(M,J,g)$ is K\"ahler (and therefore $\nabla^{\varepsilon,\rho}=\nabla^g$ for all $(\varepsilon,\rho)$), or
    \item[(b)] $(M,J,g)$ is Vaisman and $\rho=0$.
\end{enumerate}
\end{theorem}

\ 

We point out that on a Vaisman manifold $(M,J,g)$ with Lee form $\theta$ the line $\{\nabla^{\varepsilon, 0}\mid \varepsilon\in \R\}$ of metric connections goes through $\nabla^g$ (for $\varepsilon=0$) and $\nabla^b$ (for $\varepsilon=1/2)$ and, moreover, $\theta$ is parallel with respect to each one of them. 

\begin{corollary}
On a Vaisman manifold $(M,J,g)$ with Lee form $\theta$ there exist infinite metric connections which respect to which $\theta$ is parallel.
\end{corollary}

\


\begin{thebibliography}{99}
		
\bibitem{Alek}
D. Alekseevski, Homogeneous Riemannian spaces of negative curvature, \textit{Math. USSR Sb.} \textbf{25} (1975), 87--109.
		
\bibitem{AHK}
D. Alekseevsky, K. Hasegawa and Y. Kamishima, Homogeneous Sasaki and Vaisman manifolds of unimodular Lie groups, \textit{Nagoya Math. J.} \textbf{243} (2021), 83--96.

\bibitem{AI}
B. Alexandrov and S. Ivanov, Vanishing theorems on Hermitian manifolds, \textit{Differential Geom. Appl.} \textbf{14} (2001), no. 3, 251--265.

\bibitem{AS}
W. Ambrose and I. M. Singer, A theorem on holonomy, \textit{Trans. Amer. Math. Soc.} \textbf{75} (1953), 428--443.

%\bibitem{AO-1} 
%A. Andrada and M. Origlia, Locally conformally K\"ahler structures on unimodular Lie groups, \textit{Geom. Dedicata} \textbf{179} (2015), 197--216.
		
\bibitem{AO1} 
A. Andrada and M. Origlia,  Lattices in almost abelian Lie groups with locally conformal K\"ahler or symplectic structures, \textit{Manuscripta Math.} \textbf{155} (2018), 389--417.
		
\bibitem{AO}
A. Andrada and M. Origlia, Vaisman solvmanifolds and relations with other geometric structures, \textit{Asian J. Math.}, \textbf{24} (2020), 117--146.
		
%\bibitem{BDF} 
%M. L. Barberis, I. Dotti and A. Fino, Hyper-K\"ahler quotients of solvable Lie groups, \textit{J. Geom. Phys.} \textbf{56} (2006), 691--711.
\bibitem{AOUV}
D. Angella, A. Otal, L. Ugarte and R. Villacampa, On Gauduchon connections with K\"ahler-like curvature, preprint 2018, arXiv:1809.02632. To appear in Commun. Anal. Geom.

\bibitem{AP}
D. Angella and F. Pediconi, A survey on locally homogeneous almost-Hermitian spaces, preprint 2021, arXiv:2111.14577. 

\bibitem{B}
G. Barbaro, Griffiths positivity for Bismut curvature and its behaviour along Hermitian curvature flows, \textit{J. Geom. Phys.} \textbf{169} (2021), 104323.

\bibitem{Ba}
G. Bazzoni, Vaisman nilmanifolds, \textit{Bull. Lond. Math. Soc.} \textbf{49} (2017), No. 5, 824--830.

\bibitem{Bel}
F. Belgun, On the metric structure of non-K\"ahler complex surfaces, \textit{Math. Ann.} \textbf{317} (2000), 1--40.

\bibitem{Bis} 
J.-M. Bismut, A local index theorem for non-K\"ahler manifolds, \textit{Math. Ann.} \textbf{284} (1989), 681--699.
		
%\bibitem{Bo} 
%W. M. Boothby, Some fundamental formulas for Hermitian manifolds with non-vanishing torsion, \textit{Amer. J. Math.} \textbf{76} (1954), 509--534.
		
\bibitem{CMS}
R. Cleyton, A. Moroianu and U. Semmelmann, Metric connections with parallel skew-symmetric torsion, \textit{Adv. Math.} \textbf{378} (2021) 107519.
		
\bibitem{CFL} 
L. Cordero, M. Fern\'andez and M. de L\'eon, Compact locally conformal K\"ahler nilmanifolds, \textit {Geom. Dedicata} \textbf{21} (1986), 187--192.

\bibitem{DV}
S. Deaconu and V. Vuletescu, On locally conformally K\"ahler metrics on Oeljeklaus–Toma manifolds, preprint 2022, arXiv:2202.08012.

\bibitem{D}
A. Dubickas, Nonreciprocal units in a number field with an application to Oeljeklaus–Toma manifolds, \textit{New York J. Math.} \textbf{20} (2014), 257--274.

\bibitem{FPS}
A. Fino, M. Parton, and S. Salamon, Families of strong KT structures in six dimensions, \textit{Comment. Math. Helv.}  \textbf{79} (2004), 317--340.

\bibitem{FT}
A. Fino and N. Tardini, Some remarks on Hermitian manifolds satisfying K\"ahler-like conditions, \textit{Math. Z.} \textbf{298} (2021), 49--68.

\bibitem{Ga}
P. Gauduchon, Hermitian connections and Dirac operators, \textit{Boll. Un. Mat. Ital. B (7)} \textbf{11} (1997), no. 2, suppl., 257--288. 
		
\bibitem{Hano} 
J. Hano, On Kaehlerian homogeneous spaces of unimodular Lie groups, \textit{Amer. J. Math.} \textbf{79} (1957), 885--900. 
		
\bibitem{IP}
S. Ivanov and G. Papadopoulos, Vanishing theorems and string backgrounds, \textit{Classical Quantum Gravity} \textbf{18} (2001), 1089--1110.

\bibitem{J}
D. Joyce, Compact hypercomplex and quaternionic manifolds, \textit{J. Differential Geom.} \textbf{35} (1992), no. 3, 743–761.

\bibitem{K} 
H. Kasuya, Vaisman metrics on solvmanifolds and Oeljeklaus-Toma manifolds, \textit{Bull. Lond. Math. Soc.} \textbf{45} (2013), 15--26.

\bibitem{Ki}
V. F. Kiri\v{c}enko, On homogeneous Riemannian spaces with an invariant structure tensor, \textit{Dokl. Akad. Nauk. SSSR} \textbf{252} (1980), 291--293. English translation: Soviet Math. Dokl. \textbf{21} (1980), 734--737.

\bibitem{KS}
K. Kodaira and D. C. Spencer, On deformations of complex analytic structures II, \textit{Ann. of Math. (2)} \textbf{67} (1958), 403--466. 

\bibitem{LY}
J. Li and S. T. Yau, The existence of supersymmetric string theory with torsion, \textit{J. Differential Geom.} \textbf{70} (2005), 143--181.

\bibitem{MMP}
F. Madani, A. Moroianu and M. Pilca, Conformally related K\"ahler metrics and the holonomy of lcK manifolds, \textit{J. Eur. Math. Soc.} \textbf{22} (2020), 119--149. 

\bibitem{MP}
J. C. Marrero and E. Padr\'on, Compact generalized Hopf and cosymplectic solvmanifolds and the Heisenberg group $H(n,1)$, \textit{Isr. J. Math.} \textbf{101} (1997), 189--204. 


\bibitem{Mi} 
J. Milnor, Curvatures of left invariant metrics on Lie groups, \textit{Adv. Math.} \textbf{21} (1976), 293--329.

\bibitem{MMO}
A. Moroianu, S. Moroianu and L. Ornea, Locally conformally K\"ahler manifolds with holomorphic Lee field, \textit{Differential Geom. Appl.} \textbf{60} (2018), 33--38.

\bibitem{OT} 
K. Oeljeklaus and M. Toma, Non-K\"ahler compact complex manifolds associated to number fields, \textit{Ann. Inst. Fourier (Grenoble)} \textbf{55} (2005), 161--171.

\bibitem{OUV}
A. Otal, L. Ugarte and R. Villacampa, Invariant solutions to the Strominger system and the heterotic equations of motion, \textit{Nuclear Physics B} \textbf{920} (2017), 442--474.

\bibitem{Par1}
M. Parton, Explicit parallelizations on products of spheres and Calabi-Eckmann structures, \textit{Rend. Ist. Mat. Univ. Trieste} \textbf{35} (2003), 61--67. 
		
\bibitem{Par2}
M. Parton, Explicit parallelizations on products of spheres, preprint 2000, arXiv:math.DG/0009156.

\bibitem{Sam}
H. Samelson, A class of complex-analytic manifolds, \textit{Port. Math.} \textbf{12} (1953), 129–132.

\bibitem{Sc}
N. Schoemann, Almost Hermitian structures with parallel torsion, \textit{J. Geom. Phys.} \textbf{57} (2007), 2187--2212.

\bibitem{Se}
K. Sekigawa, Notes on homogeneous almost Hermitian manifolds, \textit{Hokkaido Math. J.} \textbf{7} (1978), 843--872.

\bibitem{St}
A. Strominger, Superstrings with torsion, \textit{Nuclear Phys. B} \textbf{274} (1986), 253--284.

\bibitem{V}
I. Vaisman, Locally conformal K\"ahler manifolds with parallel Lee form, \textit{Rend. Mat.} \textbf{12} (1979), 263--284.

\bibitem{Vu} 
V. Vuletescu, LCK metrics on Oeljeklaus-Toma manifolds versus Kronecker’s theorem, \textit{Bull. Math. Soc. Sci. Math. Roumanie} \textbf{57} (2014), no. 2, 225--231.

\bibitem{WYZ}
Q. Wang, B. Yang and F. Zheng, On Bismut flat manifolds, \textit{Trans. Amer. Math. Soc.} \textbf{373} (2020), 5747--5772.

\bibitem{ZZ}
Q. Zhao and F. Zheng, Strominger connection and pluriclosed metrics, preprint 2019, arXiv: 1904.06604.
\end{thebibliography}
\end{document}